\documentclass[11pt,reqno,tbtags]{amsart}

\usepackage{amsthm,amssymb,amsfonts,amsmath}
\usepackage{amscd} 
\usepackage{mathrsfs}
\usepackage[numbers, sort&compress]{natbib} 
\usepackage{subcaption, graphicx}
\usepackage{vmargin}
\usepackage{setspace}
\usepackage{paralist}
\usepackage[pagebackref=true]{hyperref}
\usepackage{color}
\usepackage[normalem]{ulem}
\usepackage{stmaryrd} 
\usepackage[refpage,noprefix,intoc]{nomencl} 
\usepackage{tcolorbox}
\usepackage{bbm,mathtools,tikz}
\usepackage{MnSymbol,wasysym}
\providecommand{\R}{}
\providecommand{\Z}{}
\providecommand{\N}{}

\renewcommand{\R}{\mathbb{R}}
\renewcommand{\Z}{\mathbb{Z}}
\renewcommand{\N}{{\mathbb N}}

\renewcommand{\P}{\mathbb{P}}
\newcommand{\D}{\mathbb{D}}

\def \W {\mathsf{W}}

\newcommand{\E}[1]{{\mathbf E}\left[#1\right]}

\newcommand{\p}{\mathbf P}

\newcommand{\I}[1]{{\mathbbm 1}_{#1}}
\newcommand{\set}[1]{\left( #1 \right)}
\newcommand{\Cprob}[2]{\mathbf{P}\set{\left. #1 \; \right| \; #2}} 
\newcommand{\probC}[2]{\mathbf{P}\set{#1 \; \left|  \; #2 \right. }}

\newcommand\cA{\mathcal A}

\newcommand\cD{\mathcal D}

\newcommand\cS{{\mathcal S}}
\newcommand\cT{{\mathcal T}}

\usepackage{crossreftools,booktabs}

\makeatletter
\newcommand{\optionaldesc}[2]{%
  \phantomsection
  #1\protected@edef\@currentlabel{#1}\label{#2}%
}
\makeatother

\newcommand{\eqdist}{\ensuremath{\stackrel{\mathrm{d}}{=}}}

\providecommand{\eps}{}
\renewcommand{\eps}{\varepsilon}
\providecommand{\ora}[1]{}
\renewcommand{\ora}[1]{\overrightarrow{#1}}
\DeclareRobustCommand{\SkipTocEntry}[5]{} 

\newcommand{\sequence}[1]{\mathrm{#1}}
\newcommand{\tree}{\mathrm{t}}

\newcommand{\height}{\mathsf{H}}
\newcommand{\rtree}{\mathrm{T}}
\newcommand{\rseq}{\mathrm{V}}

\newcommand{\dseq}{\mathrm{d}}

\renewcommand{\H}{\mathsf{Height}}
\newcommand{\Tn}{\rtree_{n}}

\newtheorem{thm}{Theorem}
\newtheorem{lem}[thm]{Lemma}
\newtheorem{prop}[thm]{Proposition}
\newtheorem{cor}[thm]{Corollary}

\newtheorem{rk}[thm]{Remark}

\makenomenclature										
\graphicspath{ {./images/} }
\usepackage{wrapfig}

\numberwithin{equation}{section}
\numberwithin{thm}{section}

\begin{document}

\title{Critical trees are neither too short nor too fat} 
\author{Louigi Addario-Berry}
\address{Department of Mathematics and Statistics, McGill University, Montr\'eal, Canada}
\email{louigi.addario@mcgill.ca}

\author{Serte Donderwinkel}
\address{Bernoulli Institute, University of Groningen, Nijenborgh~9, 9747 AG Groningen, Netherlands and CogniGron (Groningen Cognitive Systems and Materials Center), University of Groningen, Nijenborgh~4, 9747 AG Groningen, Netherlands}
\email{s.a.donderwinkel@rug.nl}

\author{Igor Kortchemski}
\address{CNRS, CMAP, École polytechnique, Institut Polytechnique de Paris, 91120 Palaiseau, France \& ETH Zürich, Rämistrasse 101, 8092 Zurich Switzerland}
\email{igor.kortchemski@math.cnrs.fr}

\begin{abstract} 
We establish lower tail bounds for the height, and upper tail bounds for the width, of critical size-conditioned Bienaym\'e trees. Our bounds are optimal at this level of generality. We also obtain precise asymptotics for offspring distributions within the domain of attraction of a Cauchy distribution, under a local regularity assumption. Finally, we pose some questions on the possible asymptotic behaviours of the height and width of critical size-conditioned Bienaym\'e trees.
\end{abstract} 

\maketitle


\section{\bf Introduction}

For a rooted tree $\tree$, the {\em size}  $|\tree|$ of $\tree$ is its number of vertices, the {\em height} $\H(\tree)$ is the greatest distance of any vertex of $\tree$ from the root, and the {\em width} $\mathsf{Width}(\tree)$ is the greatest number of nodes in any single generation of $\tree$. 

Given a probability measure $\mu$ on $\Z_{+}$ (hereafter referred to as  an \emph{offspring distribution}), write $\rtree=\rtree(\mu)$ for a $\mu$-Bienaym\'e tree, i.e., the family tree of a $\mu$-distributed branching process.\footnote{Also known as {\em Galton--Watson} or {\em Bienaym\'e--Galton--Watson} trees. Here we adopt the nomenclature proposed in \cite{ABHK21}.} 
For $n \in \N$ with $\p(|\rtree|=n)>0$, write $\rtree_n=\rtree_n(\mu)$ for a random tree distributed as $\rtree$ conditioned to have size $n$. We prove the following lower bound on $\H(\rtree_n)$, which holds for all  offspring distributions $\mu$ that are {\em critical}, meaning that they have mean~$1$.
\begin{thm}\label{thm:height_critical0}
Let $\mu$ be any critical offspring distribution. Then
\[
\mathsf{Height}( \rtree_{n}) = \omega_{\p}(\ln n)
\qquad \text{and} \qquad  
\mathsf{Width}(\rtree_n)=o_{\p}(n).
\]  
\end{thm}

The notation $X_n = \omega_{\p}(Y_n)$ means that $X_n/Y_n \to \infty$ in probability as $n \to \infty$; in the theorem (and implicitly in the sequel) the limit is along $n$ for which $\p(|\rtree|=n)>0$.

The proof relies on the powerful Foata-Fuchs bijection, which establishes a connection between trees and sequences. This bijection can be reformulated as a line-breaking construction  and we essentially show that for a critical offspring distribution, the longest path used to build the tree  has length  $\omega_{\p}(\ln n)$.

The bounds in Theorem~\ref{thm:height_critical0} are optimal for the class of critical offspring distributions; this is the content of our second theorem. 

\begin{thm}\label{thm:critical_short_trees0}
Let $f(n)\to \infty$. There exists a critical offspring distribution $\mu$ such that
\[
\limsup_{n\to \infty} \P\left(\mathsf{Height}(\rtree_{n})<f(n)\ln n\right)=1
\qquad \text{and} \qquad  
\limsup_{n\to \infty} \P\left(\mathsf{Width}(\rtree_{n})> n/f(n)\right)=1.
\]
\end{thm}

Theorems \ref{thm:height_critical0} and \ref{thm:critical_short_trees0} contribute to the recent line of work devoted to obtaining universal bounds for the height and the width of size-conditioned Bienaym\'e trees without any regularity assumptions on the offspring distribution $\mu$, see \cite{Add19,ABHK21,AD22}. 
 
Asymptotics for the tail of $\H(\rtree)$ have first been obtained by Kolmogorov \cite{Kol38} for subcritical (mean less than $1$) and critical finite variance offspring distributions (see also \cite[Eq.~(9.5)~and~10.8]{Har63}). For subcritical offspring distributions $\mu$, the finite variance condition required by Kolmogorov was relaxed by Heathcote, Seneta \& Vere-Jones \cite{HSV67} to the condition $\sum_{j \geq 1} j  \ln(j) \mu_{j} < \infty$. For critical offspring distributions $\mu$, the finite variance condition has been lifted by Slack \cite{Sla68} to the condition that $\mu$ belongs to the domain of attraction of an $\alpha$-stable distribution with $\alpha \in (1,2]$; see also \cite{BS71}. The case $\alpha=1$ has been considered in  \cite{Sze76,NW07,LS08}.

All the work cited in the previous paragraph concerns unconditioned branching processes (or equivalently Bienaym\'e trees). However, 
following the pioneering work of Aldous \cite{Ald93} on scaling limits of random trees, much effort has been devoted to obtaining limit theorems for $\H(\Tn)$. Aldous' result implies that when $\mu$ is critical and has finite variance, $\H(\Tn)/\sqrt{n}$ converges in distribution to a constant times the supremum of the normalized Brownian excursion. This also holds for supercritical offspring distributions by exponential tilting. Tail bounds on $\H(\Tn)/\sqrt{n}$ which hold uniformly in $n$ were obtained in \cite{ADS13}.

Say that a function $f:\R_{+} \to \R_{+}$ is slowly varying (at infinity) if $f( cx )/f(x) \to 1$ as $x \to \infty$ for all $c > 0$. When $\mu$ is critical and belongs to the domain of attraction of a stable distribution with index in $\alpha \in (1,2]$, it is known \cite{Duq03,Kor13} that for a certain slowly varying function $\Lambda$, $\H(\Tn)$ divided by $\Lambda(n) n^{1-1/\alpha}$ converges to a constant multiple of the supremum of the so-called {\em normalised excursion of the $\alpha$-stable height process}, introduced in \cite{LL98}. The work \cite[Theorem 1.5]{DW17} describes the asymptotic behaviour of the tail of this supremum. 
Tail bounds on $\H(\Tn)/(\Lambda(n) n^{1-1/\alpha})$ which hold uniformly in $n$, and which match the asymptotic tail behaviour found in \cite{DW17}, were proved in \cite{Kor17}.

When $\mu$ is subcritical and satisfies $\mu_n \sim C/n^{\beta}$ with $\beta>2$, then $\H(\Tn)/\ln(n)$ converges in probability to a constant \cite{Kor15}. 
This in particular shows that the criticality assumption in Theorem \ref{thm:height_critical0} is necessary. 
The subcritical, heavy-tailed regime was first studied by Jonsson \& Stef\'ansson \cite{JS11}, who identified a condensation phenomenon whereby the tree $\Tn$ with high probability contains a single vertex of degree linear in $n$ and in particular has width $\Theta(n)$.

Our next result is a new limit theorem for $\mathsf{Height}(\rtree_{n})$ in a specific regime. Say a sequence $(a_n)_{n\geq 1}\in \R_{+}^{\N}$ is slowly varying if the function $f(x)\coloneqq a_{\lfloor x\rfloor}$ is slowly varying.
An offspring distribution $\mu$ is said to be in the {\em domain of attraction of a Cauchy random variable} if  $(n\mu[n,\infty)),n \ge 1)$ is slowly varying; see \cite[IX.8, Eq. (8.14)]{Fel71}. We make the following slightly stronger assumption: 
\begin{equation}
   \label{eq:hypmu}
  \mu_n=  \frac{L(n)}{n^{2}}  \qquad \text{and} \qquad  \mu \text{ is critical},
  \tag{$\textrm{H}_\mu$}   \end{equation}
  where $L$ is slowly varying. This can be viewed as a ``local'' Cauchy condition. In this framework, some properties of $\rtree_{n}$ have been studied in \cite{KR19}, motivated by applications to random maps.  An interesting condensation phenomenon appears for such trees: with probability tending to $1$ as $n \rightarrow \infty$, the maximal degree in $\Tn$ dominates the others, while the local limit of $\Tn$ is locally finite. This means that vertices with maximal degrees ``escape to infinity'' and disappear in the local limit.

It turns that a limit theorem holds for $\mathsf{Height}(\rtree_{n})$ and $\mathsf{Width}(\rtree_{n})$, with a certain  scaling sequence which we now define. Assume that $\mu$ satisfies \eqref{eq:hypmu} and let $Y$ be $\mu$-distributed. Then fix any sequence $(a_n,n \ge 1)$ such that $n \p(Y \ge a_n) \to 1$ as $n \to \infty$, and set 
\begin{equation}
\label{eq:hn}h_n = \int_1^{n \E{Y\I{Y \ge a_n}}} \frac{1}{ x \E{Y \I{Y \ge x}}} \mathrm{d}x\, .
\end{equation}
\begin{thm}
\label{thm:heightCauchy0}
Assume that  $\mu$ satisfies \eqref{eq:hypmu}. Then the following convergences holds in probability:
\[\frac{\mathsf{Height}(\rtree_{n})}{h_{n}} \quad \xrightarrow[n\to\infty]{} \quad 1 \qquad \textrm{and} \qquad \frac{\mathsf{Width}(\rtree_{n})}{n \E{Y\I{Y \ge a_n}}} \quad \xrightarrow[n\to\infty]{} \quad 1.\]
\end{thm}
We will later prove that, for $a_n$ and $h_n$ as above, $h_{n}\cdot n \E{Y\I{Y \ge a_n}}=O(n \log n)$ and so Theorem~\ref{thm:heightCauchy0} has the following corollary.
\begin{cor}\label{cor:hw}
Assume that  $\mu$ satisfies \eqref{eq:hypmu}. Then 
$\mathsf{Height}(\rtree_{n}) \cdot \mathsf{Width}(\rtree_{n}) = O_{\p}({n \ln(n)})$.
\end{cor}
This answers a question from \cite{Add19} in the affirmative, for the class of distributions satisfying \eqref{eq:hypmu}. For a further discussion of this question, see Section~\ref{sec:future}. 

We will later see that $\E{Y\I{Y \ge a_n}}$ and $(h_{n})$ are slowly varying. The fact that such a limit theorem holds for $\mathsf{Height}(\rtree_{n})$ is surprising for multiple reasons. First, $\rtree_{n}$ does not have any nontrivial scaling limits \cite{KR19}. Second, in contrast to the other cases when limit theorems for $\mathsf{Height}(\rtree_{n})$ hold, here the asymptotic behavior of the tail of $\mathsf{Height}(\rtree)$ is not universal, in the following sense: if $\mu$ and $\mu'$ are two offspring distributions satisfying respectively $\textrm{H}_\mu$ and $\textrm{H}_{\mu'}$ such that $\mu_n \sim \mu'_n$ as $n \rightarrow \infty$, then in general it is not true that $\p(\mathsf{Height}(\rtree) \geq x)  \sim \p(\mathsf{Height}(\rtree') \geq x)$ as $ x \rightarrow \infty$, where $\rtree'$ is a $\mu'$ Bienaym\'e tree  (this is true when $L(x)=o(1/\ln(x)^{2})$ but false in general, see \cite[Sec.~3.2]{NW07}). However, it is true that $h_{n} \sim h'_{n}$ as $n \rightarrow \infty$, where $h'_{n}$ is defined like $h_{n}$ but replacing $\mu$ by $\mu'$: in a certain sense, the effect of size-conditioning restores asymptotic universality. This is the reason why the proof of Theorem \ref{thm:heightCauchy0} is  rather delicate.

Let us also mention that if one chooses another sequence $\tilde{a}_n$ such that $n \p(Y \ge \tilde{a}_n) \to 1$ as $n \to \infty$ and defines $\tilde{h}_{n}$ accordingly, then $h_{n} \sim \tilde{h}_{n}$ as $n \rightarrow \infty$. In particular, for concreteness, one could take  $a_{n}= \inf\{u >0: \p(Y \geq u) \leq 1/n\} $.

The main strategy to prove Theorem \ref{thm:heightCauchy0} is to combine results of \cite{KR19} concerning the structure of $\rtree_{n}$ with bounds from \cite{NW07} on the tail of $\mathsf{Height}(\rtree)$.

\paragraph{\textbf{Examples.}} Let us give some explicit examples, inspired by \cite{NW07}; see Section~\ref{ssec:examples} for details.

--  For $L(n) \sim \beta \ln(n)^{-1-\beta}$ with $\beta>0$, we have \[
 h_{n}  \quad \sim \quad \frac{1}{1+\beta} \ln(n)^{1+\beta}
 \]
 
--  Set $\ln_{(1)}(x)=\ln(x)$ and for every $k \geq 1$ define recursively $\ln_{(k+1)} (x) = \ln(\ln_{(k)}(x))$. For
 \[
 L(n) \quad \sim \quad  \frac{1}{(\ln_{(k)}(n))^{2}} \prod_{i=1}^{k-1} \frac{1}{\ln_{(i)}(n) } \]
 with $k \geq 2$, we have 
 \[
 h_{n}  \quad \sim \quad \ln(n) \ln_{(k)}(n)
 \]

-- For $L(n) \sim  \frac{1-\beta}{\beta} \ln(n)^{-\beta} e^{-\ln(n)^{\beta}} $ with $\beta \in (0,1)$, letting $k \geq 0$ be the smallest integer such that $ \frac{k}{k+1} \leq \beta < \frac{k+1}{k+2}$, we have
 \[
 h_{n}  \quad \sim \quad \exp \left( \sum_{i=0}^{k}  P_{i}(\beta) \ln(n)^{\beta-i(1-\beta)} \right),
 \]
for some polynomials $(P_i,0 \le i \le k)$. 

The final example demonstrates the wide range of asymptotic behaviour which can appear, and also  illustrates one subtlety in the definition of $h_n$. Indeed, in the final example, taking the upper limit in the integral in \eqref{eq:hn} to be $n \E{Y\I{Y \ge n}}$, or even $n$, rather than $n \E{Y\I{Y \ge a_n}}$, would change the asymptotic behaviour of $h_n$; this is not the case in the first two examples. (Again, see Section~\ref{ssec:examples}  for the details.)
 
 \paragraph{\bf Outline.}
 
First, in Section \ref{sec:trees}, we discuss the encoding of (random) planar trees by (random) lattice paths; such encodings are used throughout the remainder of the paper. Then, in Section \ref{sec:cauchy_trees}, we discuss some general results which describe the structure of large Cauchy--Bienaym\'e trees. This allows us to prove  prove the height and width bounds of \ref{thm:heightCauchy0}, as well as Corollary \ref{cor:hw}. We also provide further detail about the examples from above. In Section \ref{sec:height_critical} we use lattice path encodings to derive information about the degree sequences of large critical Bienaym\'e trees and to prove the width bound from Theorem~\ref{thm:height_critical0}. We also introduce an additional tool -- the Foata-Fuchs bijection for labeled trees with given degrees \cite{cayley_us,FoataFuchs}  --  which we combine with the previously derived information about the degree sequences of large critical Bienaym\'e trees in order to prove Theorem \ref{thm:height_critical0}. 
{In Section \ref{sec:quite_short_quite_fat} we prove Theorem \ref{thm:critical_short_trees0} using a stochastic domination result for the height of labeled trees with given degrees \cite{AD22} and an explicit construction of offspring distributions that yield trees with large width and small height.} Finally,~Section~\ref{sec:future} contains some open problems and conjectures related to the possible asymptotic behaviour of conditioned critical Bienaym\'e trees (with a particular focus on the Cauchy case).

 {\paragraph{\bf Acknowledgements.} This research was mainly done during the Seventeenth Annual Workshop on Probability and Combinatorics at McGill University’s Bellairs Institute in Holetown, Barbados. We are grateful to Bellairs Institute for its hospitality. LAB acknowledges the financial support of NSERC and the CRC program. SD acknowledges the financial support of the CogniGron research center and the Ubbo Emmius Funds (Univ. of Groningen). We also thank the referee for a very careful reading.}

\tableofcontents

\section{\bf Bienaym\'e trees and their coding}
\label{sec:trees}

\subsection{Plane trees}
\label{sssec:planetrees}

We define plane trees according to Neveu's formalism \cite{Nev86}. Let $\N = \{1, 2, \dots\}$ be the set of positive integers, and consider the set of labels $\mathcal{U} = \bigcup_{n \ge 0} \N^n$ (where by convention $\N^0 = \{\varnothing\}$). For every $v = (v_1, \dots, v_n) \in \mathcal{U}$, the length of $v$ is $\|v\| = n$. For $u,v\in \mathcal{U}$, we let $uv$ be the concatenation of $u$ and $v$.

Then, a (locally finite) \emph{plane tree} is a nonempty subset $\tree \subset \mathcal{U}$ satisfying the following conditions. First, $\varnothing \in \tree$ ($\varnothing$ is called the \textit{root vertex} of the tree). Second, if $v=(v_1,\ldots,v_n) \in \tree$ with $n \ge 1$, then $(v_1,\ldots,v_{n-1}) \in \tree$ ($(v_1,\ldots,v_{n-1})$ is called the \textit{parent} of $v$ in $\tree$). Finally, if $v \in \tree$, then there exists an integer $k_v(\tree) \ge 0$ such that $(v_1,\ldots,v_n,i) \in \tree$ if and only if $1 \le i \le k_v(\tree)$.  The quantity $k_v(\tree)$ is the \textit{number of children} of $v$ in $\tree$; we also refer to $k_v(\tree)$ as the \emph{degree} of $v$ in $\tree$. The plane tree $\tree$ may be seen as a genealogical tree in which the individuals are the vertices $v\in\tree$. 

 For $v,w \in \tree$, we let  $\llbracket v, w \rrbracket$ be the vertices belonging to the shortest path from $v$ to $w$ in $\tree$. We use $\llbracket v, w \llbracket$ for the  set obtained from $\llbracket v, w \rrbracket$ by excluding $w$. We also let $|\tree|$ be the total number of vertices (that is, the size) of the plane tree $\tree$. 

\subsection{Coding trees by lattice paths}
\label{sssec:coding}
Let $\tree$ be a {finite} plane tree. We associate with every ordering $\varnothing=u_{0}\prec u_{1}\prec \cdots \prec u_{|\tree|-1}$ of  the vertices of $\tree$  the path  $\W(\tree)=( \W_n(\tree) : 0 \leq n < |\tree|)$,  defined by $ \W_0(\tree)=0$, and $ \W_{n+1}(\tree)= \W_{n}(\tree)+k_{u_n}(\tree)-1$ for every $0 \leq n < |\tree|$.

We will use two different orderings of the vertices of a tree:
\begin{enumerate}
\item[(i)] the lexicographical ordering, where  $v \prec w$ if there exists $u \in \mathcal{U}$ such that $v = u(v_1, \dots, v_n)$, $w = u(w_1, \dots, w_m)$ and $v_1 < w_1$;
\item[(ii)] the breadth-first search ordering, where $v \prec w$ if either $\|v\| <\|w\|$, or $\|v\|=\|w\|$ and $v=u(v_1, \dots, v_m)$ and $w=u (w_1, \dots, w_m)$ with $v_{1}<w_{1}$.
\end{enumerate}

Denote by $\W^{\mathsf{lex}}(\tree)$ and $\W^{\mathsf{bfs}}(\tree)$ the paths constructed by using respectively the lexicographical and  breadth-first search ordering of the vertices of $ \tree$. We refer to both $\W^{\mathsf{lex}}$ and $\W^{\mathsf{bfs}}$ as {\em exploration processes}; we call $\W^{\mathsf{lex}}(\tree)$ the {\em {\L}ukasiewicz path} of $ \tree$ and $\W^{\mathsf{bfs}}(\tree)$ the {\em breadth-first queue process} of~$\tree$.

The following deterministic fact will be useful.
Let $\tree$ be a finite tree and order the vertices of $\tree$ using the breadth-first search ordering as $u_0,\dots,u_{|\tree|-1}$. Then for every $1 \le i \le |\tree|$, writing $\ell=\|u_{i-1}\|$, then 
$\W^{\mathsf{bfs}}_{{i}}(\tree)$ is equal to the total number of children of generation-$\ell$ predecessors of $u_{i-1}$ plus the number of generation-$\ell$ successors of $u_{i-1}$ (where ``predecessor'' and ``successor'' are relative to the breadth-first search order).

As a consequence,
\begin{equation}
\label{eq:width}
\Delta(\tree) \leq \textsf{Width}(\tree) \leq \max(\W^{\mathsf{bfs}}_i(\tree), 0 \le i < |t|), 
\end{equation}
where we define $\Delta(\tree)=\max_{u \in \tree} k_{u}(\tree)$ to be the maximal number of children of a vertex.

Set $[n]=\{1,2,\dots, n\}$. A {\em discrete bridge of length $n$} is a lattice path $\mathbf{s}=(s_0,\ldots,s_n)$ with $s_0=0$ and $s_n=-1$ and $s_{i}-s_{i-1} \ge -1$ for $i \in [n]$; it is a {\em discrete excursion} if additionally $s_i \ge 0$ for $i \in [n-1]$.

For either $\ast \in \{\mathsf{lex},\mathsf{bfs}\}$ and any $n \ge 1$, the map
\[
\tree \mapsto \W^{\ast}(\tree)
\]
induces an invertible function with domain the set of plane trees with $n$ vertices and range the set of discrete excursions of length $n$. 

The {\em Vervaat transform} is an $n$-to-$1$ map from discrete bridges of length $n$ to discrete excursions of length n which is defined as follows. Let $\mathbf{s}=(s_i : 0 \leq i \leq n)$ be a discrete bridge, and for $1 \le i \le n$ let $x_i=s_i-s_{i-1}$. 
Also, let
\[
m=m(\mathbf{s})\coloneqq\min\{0\leq i \leq n : s_i=\min\{s_j : 0\leq j \leq n\}\}.
\]
be the first time at which $(s_i : 0\leq i \leq n)$ reaches its overall minimum. Then the Vervaat transform $\mathcal{V}(\mathbf{s}) \coloneqq (\mathcal{V}(\mathbf{s})_i : 0\leq i \leq n)$ of $\mathbf{s}$ is the walk obtained by reading the increments $(x_1,\ldots, x_n)$ from left to right in cyclic order, started from $m$. Namely, 
\[
\mathcal{V}(\mathbf{s})_0=0 \quad \text{and} \quad \mathcal{V}(\mathbf{s})_{i+1}-\mathcal{V}(\mathbf{s})_{i}= x_{m+i \mod[n]}, \quad 0\leq i < n.
\] 
See Figure \ref{fig:vervaat} for an illustration. 

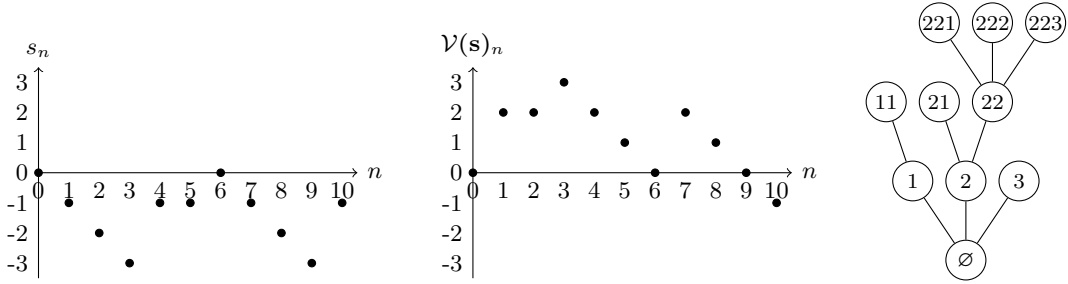
\begin{figure}[h!]
\begin{center}
\begin{tikzpicture}[scale = 0.4]
\draw[->] (0,-3.5) -- (0,3.5);
\draw[->] (0,0) -- (10.5,0);

\node[above,font=\footnotesize] at (0,3.5) {$s_n$};
\node[right,font=\footnotesize] at (10.5,0) {$n$};

\node[draw,circle,fill,inner sep = 1pt] at (0,0) {};
\node[draw,circle,fill,inner sep = 1pt] at (1,-1) {};
\node[draw,circle,fill,inner sep = 1pt] at (2,-2) {};
\node[draw,circle,fill,inner sep = 1pt] at (3,-3) {};
\node[draw,circle,fill,inner sep = 1pt] at (4,-1) {};
\node[draw,circle,fill,inner sep = 1pt] at (5,-1) {};
\node[draw,circle,fill,inner sep = 1pt] at (6,0) {};
\node[draw,circle,fill,inner sep = 1pt] at (7,-1) {};
\node[draw,circle,fill,inner sep = 1pt] at (8,-2) {};
\node[draw,circle,fill,inner sep = 1pt] at (9,-3) {};
\node[draw,circle,fill,inner sep = 1pt] at (10,-1) {};

\node[left,font=\footnotesize] at (0,-1) {-1};
\node[left,font=\footnotesize] at (0,-2) {-2};
\node[left,font=\footnotesize] at (0,-3) {-3};
\node[left,font=\footnotesize] at (0,0) {0};
\node[left,font=\footnotesize] at (0,1) {1};
\node[left,font=\footnotesize] at (0,2) {2};
\node[left,font=\footnotesize] at (0,3) {3};

\node[below,font=\footnotesize] at (0,0) {0};
\node[below,font=\footnotesize] at (1,0) {1};
\node[below,font=\footnotesize] at (2,0) {2};
\node[below,font=\footnotesize] at (3,0) {3};
\node[below,font=\footnotesize] at (4,0) {4};
\node[below,font=\footnotesize] at (5,0) {5};
\node[below,font=\footnotesize] at (6,0) {6};
\node[below,font=\footnotesize] at (7,0) {7};
\node[below,font=\footnotesize] at (8,0) {8};
\node[below,font=\footnotesize] at (9,0) {9};
\node[below,font=\footnotesize] at (10,0) {10};
\end{tikzpicture}
\quad 
\begin{tikzpicture}[scale = 0.4]
\draw[->] (0,-3.5) -- (0,3.5);
\draw[->] (0,0) -- (10.5,0);

\node[above,font=\footnotesize] at (0,3.5) {$\mathcal{V}(\mathbf{s})_n$};
\node[right,font=\footnotesize] at (10.5,0) {$n$};

\node[draw,circle,fill,inner sep = 1pt] at (0,0) {};
\node[draw,circle,fill,inner sep = 1pt] at (1,2) {};
\node[draw,circle,fill,inner sep = 1pt] at (2,2) {};
\node[draw,circle,fill,inner sep = 1pt] at (3,3) {};
\node[draw,circle,fill,inner sep = 1pt] at (4,2) {};
\node[draw,circle,fill,inner sep = 1pt] at (5,1) {};
\node[draw,circle,fill,inner sep = 1pt] at (6,0) {};
\node[draw,circle,fill,inner sep = 1pt] at (7,2) {};
\node[draw,circle,fill,inner sep = 1pt] at (8,1) {};
\node[draw,circle,fill,inner sep = 1pt] at (9,0) {};
\node[draw,circle,fill,inner sep = 1pt] at (10,-1) {};

\node[left,font=\footnotesize] at (0,-1) {-1};
\node[left,font=\footnotesize] at (0,-2) {-2};
\node[left,font=\footnotesize] at (0,-3) {-3};
\node[left,font=\footnotesize] at (0,0) {0};
\node[left,font=\footnotesize] at (0,1) {1};
\node[left,font=\footnotesize] at (0,2) {2};
\node[left,font=\footnotesize] at (0,3) {3};

\node[below,font=\footnotesize] at (0,0) {0};
\node[below,font=\footnotesize] at (1,0) {1};
\node[below,font=\footnotesize] at (2,0) {2};
\node[below,font=\footnotesize] at (3,0) {3};
\node[below,font=\footnotesize] at (4,0) {4};
\node[below,font=\footnotesize] at (5,0) {5};
\node[below,font=\footnotesize] at (6,0) {6};
\node[below,font=\footnotesize] at (7,0) {7};
\node[below,font=\footnotesize] at (8,0) {8};
\node[below,font=\footnotesize] at (9,0) {9};
\node[below,font=\footnotesize] at (10,0) {10};
\end{tikzpicture}
\quad 
\begin{tikzpicture}[scale = 0.7,
sommet/.style = {draw,circle, font=\scriptsize,inner sep=0,minimum size=15pt},
etiquete/.style = {font = \scriptsize}]

\node[sommet] (0) at (0,0) {$\varnothing$};
\node[sommet] (1) at (-1,1.5) {$1$};
\node[sommet] (2) at (0,1.5) {$2$};
\node[sommet] (3) at (1,1.5) {$3$};
\node[sommet] (11) at (-1.5,3) {$11$};
\node[sommet] (21) at (-0.5,3) {$21$};
\node[sommet] (22) at (0.5,3) {$22$};
\node[sommet] (221) at (-0.5,4.5) {$221$};
\node[sommet] (222) at (0.5,4.5) {$222$};
\node[sommet] (223) at (1.5,4.5) {$223$};

\draw (0) -- (1);
\draw (0) -- (2);
\draw (0) -- (3);
\draw (1) -- (11);
\draw (2) -- (21);
\draw (2) -- (22);
\draw (22) -- (221);
\draw (22) -- (222);
\draw (22) -- (223);

\end{tikzpicture}
\end{center}
\caption{Left: A walk $\mathbf{s}=(s_0,\ldots,s_{10})$ with $s_{10}=-1$. Centre: the Vervaat transform $\mathcal{V}(\mathbf{s})$. Right: the plane tree $\tree$ with $\W^{\mathsf{bfs}}(\tree)=\mathcal{V}(\mathbf{s})$.}
\label{fig:vervaat}
\end{figure}

Let $\mu$ be a critical or subcritical offspring distribution. The Bienaym\'e measure with offspring distribution $\mu$ is the probability measure $\P_\mu$  on plane trees that is characterized by
\begin{equation}\label{eq:def_GW}
\P_\mu(\tree) = \prod_{u \in \tree} \mu_{k_u(\tree)}
\end{equation}
for every finite plane tree $\tree$ (see~\cite[Prop.~1.4]{LG05}); in particular, $\rtree(\mu)$ has law $\P_\mu$.

The following result relates the exploration processes of $\rtree(\mu)$ to a random walk (see \cite[Proposition 1.5]{LG05} and/or \cite[Chapter 5 Exercise 1]{MR2245368}). From this point on, $(X_{i} : i \geq 1)$ will always denote a sequence of i.i.d.\ random variables with law given by $\P(X_{1}=i)=\mu_{i+1}$ for $i \geq -1$ for some offspring distribution $\mu$, and we will always write $S_0=0$ and $S_i=X_1+\ldots+X_i$ for $i \ge 1$. Also, let $\zeta=\inf\{i \ge 1: S_i = -1\}$.
\begin{prop}
\label{prop:GWRW}
Fix a critical or subcritical offspring distribution $\mu$ with $\mu_0>0$. 
Then with $\rtree=\rtree(\mu)$, for every  $\ast \in \{\mathsf{lex},\mathsf{bfs}\}$ we have
\[
(\W^{\ast}_{i}( \rtree),0 \le i \le |\rtree|)\eqdist 
(S_i,0 \le i \le \zeta)\, .
\]
Moreover, for any $n \in \N$ the following holds. For any discrete excursion $\mathbf{w}=(w_0,\ldots,w_n)$, 
\[
\Cprob{(S_i,0 \le i \le n)=\mathbf{w}}{\zeta=n}
=\Cprob{\mathcal{V}(S_i,0 \le i \le n)=\mathbf{w}}{S_n=-1}\, .
\]
\end{prop}
Together the two identities of Proposition~\ref{prop:GWRW} imply that for $\rtree_n=\rtree_n(\mu)$, we have 
\begin{equation}\label{eq:condition_identity}
\p((\W^{\ast}_{i}( \rtree_n),0 \le i \le n)=\mathbf{s})
=
\Cprob{\mathcal{V}(S_i,0 \le i \le n)=\mathbf{w}}{S_n=-1}.
\end{equation}

The following fact follows readily from \eqref{eq:condition_identity}.
\begin{rk}
\label{rk:degrees}
The multiset of degrees in $\rtree_n$ has the same law as the conditional law of the multiset $\{X_i+1,1 \le i \le n\}$ given that $S_n=-1$.
 \end{rk}

\section{\bf The height and width of large conditioned critical Cauchy--Bienaym\'e trees}
\label{sec:cauchy_trees}

Throughout this section we assume that $\mu$ satisfies \eqref{eq:hypmu}, namely that $\mu$ is critical and $\mu_n=L(n)/n^{2}$ with $L : \R_{+} \rightarrow \R_{+}$ is slowly varying. 

\subsection{Scaling constants}\label{sec:scalingconstants} 
Let $(a_{n} : n \geq 1)$ and $(b_{n} : n \geq 1)$ be sequences such that
\begin{equation}
\label{eq:defanbn} n \p(X \geq a_{n})  \quad \xrightarrow[n\to\infty]{} \quad  1, \qquad b_{n}=-n \E{X \mathbbm{1}_{|X| \leq  a_{n}}},
\end{equation} 
where the law of $X$ is given by $\P(X=i)=\mu_{i+1}$ for $i \geq -1$. Observe that since $\mu$ is critical we have $\E{X}=0$, so that $b_{n}=n \E{X \mathbbm{1}_{|X| >  a_{n}}}$.

The main reason why  the scaling constants $(a_{n})_{n \geq 1}$ and $(b_{n})_{n \geq 1}$ appear is the following: for $(X_{i}:i \geq 1)$ a  a sequence of i.i.d.~random variables distributed as $X$ and $S_{n}=X_{1}+ \cdots+X_{n}$, then the convergence in distribution
\begin{equation}
\label{eq:cvCauchy}
 \frac{ S_{n} +b_{n}}{a_{n}}  \quad \xrightarrow[n\to\infty]{(d)} \quad  \mathcal{C}_{1}
\end{equation} 
holds, where $ \mathcal{C}_{1}$ is the random variable with Laplace transform given by $\E{e^{-\lambda \mathcal{C}_{1} } }=e^{\lambda \ln(\lambda)}$ for $\lambda>0$; $\mathcal{C}_1$ is an asymmetric Cauchy random variable with skewness $1$, see \cite[Chap. IX.8 and Eq.~(8.15) p315]{Fel71}. It is important to observe that even though $\E{S_{n}}=0$, a centering term is required in \eqref{eq:cvCauchy}.
The sequences $(a_{n})$ and $(b_{n})$ are both regularly varying of index $1$, meaning that $a_{n}/n$ and $b_{n}/n$ are slowly varying, and also $b_{n} \rightarrow  \infty$ and $a_{n}=o(b_{n})$ (see Remark \ref{rk:anbn}).

\begin{figure}[h!]
\centering
\includegraphics[scale=.5]{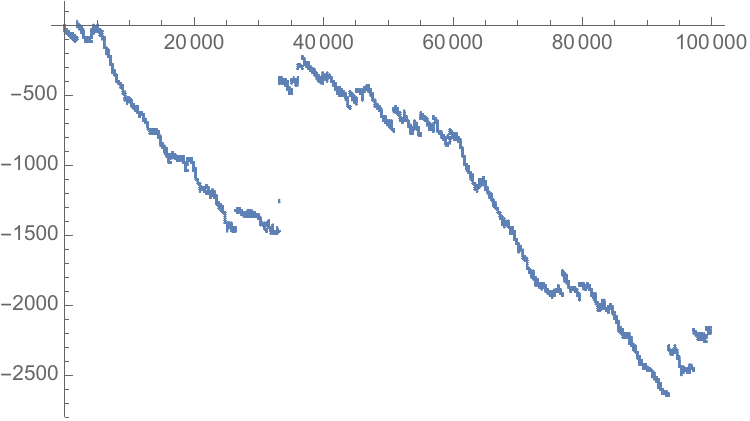}
\hfill 
\includegraphics[scale=.5]{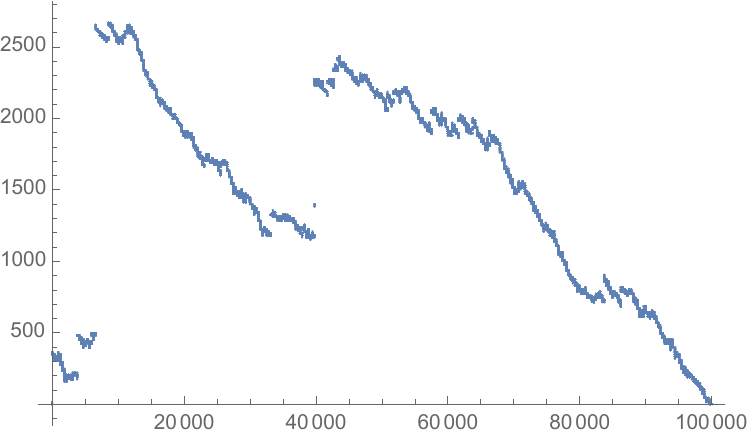}
\caption{Left: a simulation of $(S_{i})_{0 \leq i \leq n-1}$ for $n=100000$, where $(S_i,i \ge 0)$ is the random walk defined in Section~\ref{sec:scalingconstants} with $\mu_n\sim \frac{c}{n^{2} \ln(n)^{2}}$. Right: the associated discrete excursion $Z^{(n)}$ obtained by performing the Vervaat transform on $(S_0,S_1,\ldots,S_{n-1},-1)$. 
}
\label{fig:walk}
\end{figure}

\subsection{The exploration processes of a size-conditioned Cauchy--Bienaymé tree}
\label{ssec:luka}
We now describe the behavior of the exploration processes  of $ \rtree_{n}$, a $\mu$-Bienaym\'e tree conditioned to have size $n$, under the assumption \eqref{eq:hypmu}.
We will use the Vervaat transform $\mathcal{V}$ defined in Section \ref{sssec:coding}. Recall that $(S_i : i \geq 0)$ is a random walk with increments distributed as $X$, and for every $n\in\Z_+$ define the random process $Z^{(n)} \coloneqq (Z^{(n)}_{i} : 0\leq i  \leq n )$  by 
\begin{equation}\label{eqn:ZnVervaat}
	Z^{(n)} \coloneqq \mathcal{V}(S_0,S_1,\ldots,S_{n-1},-1)
\end{equation} 
when $S_{n-1} \leq 0$ and $Z^{(n)}=(0,0, \ldots,-1)$ otherwise (observe that $S_{n-1} \leq 0$ with probability tending to $1$ as $ n \rightarrow \infty$ by \eqref{eq:cvCauchy}).  
The next theorem states that both of the exploration processes of $ \rtree_{n}$ are close in total variation distance to $Z^{(n)}$ when $n$ goes to infinity.
\begin{thm}[Theorem 21 in \cite{KR19}]
\label{thm:dTVLoc}
 For every  $\ast \in \{\mathsf{lex},\mathsf{bfs}\}$, we have
\[
d_{\mathrm{TV}}\left( \mathsf{W}^{\ast}(\rtree_{n}) , Z^{(n)}\right)  \quad \xrightarrow[n\to\infty]{}  \quad 0,
\]
where $d_{\mathrm{TV}}$ denotes the total variation distance on $\R^{n+1}$ equipped with the product topology.
\end{thm}

Let us mention that compared to  \cite{KR19}  we use a different convention concerning the definition of $b_{n}$: the two definitions differ by a factor of $-1$. 

The following result on the scaling limit of the exploration processes of $ \rtree_{  n}$ implies that, in probability, the processes have the following structure: they make a big jump of size $(1+o(1))b_n$ within time $o(n)$, and then go down to $-1$ linearly with fluctuations of size $o(b_n)$. 

\begin{prop}[Proposition 24 in \cite{KR19}]\label{prop:functional_conv_W}
\label{prop:W}
For $\ast \in \{\mathsf{lex},\mathsf{bfs}\}$, extend the exploration process $\mathsf{W}^{\ast}(\rtree_{  n})$ to $[-n,n]$ by setting $\mathsf{W}^{\ast}_{-i}(\rtree_{  n})=0$ for $i=1,\dots, n$. Then for  $\ast \in \{\mathsf{lex},\mathsf{bfs}\}$ it holds that 
\[ 
\left(  \frac{\mathsf{W}^{\ast}_{\lfloor n t \rfloor}(\rtree_{  n})}{b_n}, -1\leq t \leq 1 \right) \overset{d}{\to} \left((1-t)\I{t\geq 0},-1\leq t \leq 1 \right)
\]
in $\D([-1,1],\R)$ as $n\to \infty$.
\end{prop}
It might seem artificial to extend the process to a process on $[-1,1]$, but an extension to some interval $[c,1]$ with $c <0$ is required to obtain convergence in the Skorokhod topology: indeed, the limit process has a jump instantaneously at time 0, while in $\mathsf{W}(\rtree_{  n})$ the macroscopic jump occurs at a strictly positive time.

\begin{cor}\label{cor:width}
Denote by $\mathsf{Width}(\rtree_{n})$ the width of $ \rtree_{  n}$. Then the convergence
\[\frac{\mathsf{Width}(\rtree_{n})}{b_n} \quad \xrightarrow[n\to\infty]{} \quad 1\]
holds in probability.
\end{cor}

Corollary \ref{cor:width} readily follows from \eqref{eq:width}, since by Proposition \ref{prop:W} we have the following convergence in probability:
\[\frac{\max\Delta\W^{\mathsf{bfs}}(\rtree_{n}) }{ b_{n} }  \quad \xrightarrow[n\to\infty]{} \quad 1, \qquad \frac{\max \W^{\mathsf{bfs}}(\rtree_{n})}{ b_{n} } \quad \xrightarrow[n\to\infty]{} \quad 1.\]

We now turn to the study of $\mathsf{Height}(\rtree_{n})$. From here on we will sometimes write $\mathsf{H}(\cdot)=\mathsf{Height}(\cdot)$ to make equations more readable.

\subsection{Reduction to a forest}
\label{ssec:forest}
Recall from Sec.~\ref{ssec:luka} the construction 
\[
Z^{(n)} \coloneqq \mathcal{V}(S_0,S_1,\ldots,S_{n-1},-1)
\]
when $S_{n-1} \le 0$ and $Z^{(n)}=(0,0, \ldots,-1)$ otherwise.
We denote by $ \rtree'_{n}$ the tree whose {\L}ukasiewicz path is $Z^{(n)}$
and by $v'_n$ the (lexicographically least) vertex of maximal degree in $\rtree'_{  n}$. Observe that by Proposition \ref{prop:functional_conv_W} and since $S_{n}/b_{n} \rightarrow -1$ in probability,
it holds with probability tending to $1$ as $n \rightarrow \infty$ that this vertex is unique and has degree $\Delta(\rtree'_{  n})= \vert S_{n-1} \vert$, and that $S_{n-1}<-1$.  In addition, Theorem~\ref{thm:dTVLoc} implies that $d_{\mathrm{TV}}(\rtree_{n},\rtree'_{n}) \rightarrow 0$.
 
We now introduce a decomposition of the walk $(S_i : i \geq 0)$ into excursions above its infimum. Set $\underline{S}_{k} = \min \{S_{0}, S_{1}, \ldots,S_{k}\}$ for every $k \geq 0$, let $\zeta_{k}=\inf\{i\geq 0 : S_i=-k\}$  for every $k\geq 0$, and define the excursions
 \[
 \left( \mathcal{W}^{(k)}_i : 0 \leq i \leq \zeta_{k}-\zeta_{k-1} \right) \coloneqq \left( k-1+S_{\zeta_{k-1}+i} : 0 \leq i \leq \zeta_{k}-\zeta_{k-1} \right), \quad k\geq 1.
 \] For every $k\geq 1$, we let $\tau_k$ be the tree whose {\L}ukasiewicz path is $\mathcal{W}^{(k)}$. Observe that the trees $(\tau_k)_{k \ge 1}$ are independent and $T(\mu)$-distributed. Moreover, when $S_{n-1} \leq 0$, for every $1\leq k \leq \vert S_{n-1}\vert$, $\tau_k$ is  the subtree rooted at the $k$'th child of $v'_n$ in $\rtree'_{n}$. Denote by $ \mathcal{F}^{n}_{\mathrm{L}}$ (resp.\ $ \mathcal{F}^{n}_{\mathrm{R}}$) the forest of trees rooted at the children of $\llbracket \varnothing, v'_{n} \llbracket$ and which are on the left (resp.\ right) of $\llbracket \varnothing, v'_{n} \llbracket$. Observe also that  $ \mathcal{F}^{n}_{\mathrm{R}} = \{ \tau_k : |S_{n-1}|<k \leq |\underline{S}_{n-1}| \} $ when $S_{n-1} \leq 0$. 

\begin{figure}[h!]
  \centering
  \includegraphics[width=\linewidth]{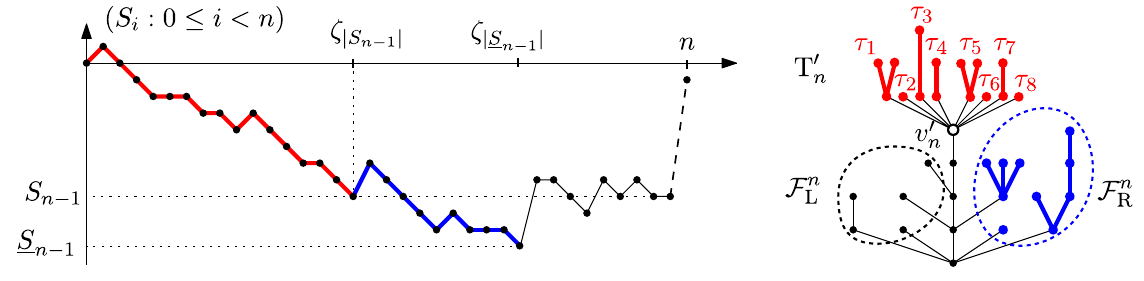}
\caption{The random walk $(S_i : 0\leq i < n)$ and the associated tree $\rtree'_{  n}$. In red, the $\vert S_{n-1} \vert $ first excursions of $S$ that encode the trees grafted above the vertex with maximal degree $v'_n$. In blue, the forest $\mathcal{F}^{n}_{\mathrm{R}}$ made of the trees grafted on the children of $\llbracket \varnothing, v'_{n}\llbracket$ on the right of $\llbracket \varnothing, v'_{n}\llbracket$. In this example, $\mathsf{H}(\mathcal{F}^{n}_{\mathrm{L}})=2$ and $\mathsf{H}(\mathcal{F}^{n}_{\mathrm{R}})=4$.}
\label{fig:tprimen}
\end{figure}

It is an immediate consequence of the above definitions that when $S_{n-1} \leq 0$,
\begin{equation}
\label{eq:boundH}
\max_{1 \leq i \leq |S_{n-1}|}\mathsf{H}(\tau_{i}) \leq \mathsf{H}(\rtree'_{  n}) \leq  \|v'_n\|+ \max\left(\max_{1 \leq i \leq |S_{n-1}|}\mathsf{H}(\tau_{i}),\mathsf{H}(\mathcal{F}^{n}_{\mathrm{L}}),\mathsf{H}(\mathcal{F}^{n}_{\mathrm{R}})\right),
\end{equation}
where by definition the height of a forest is the maximal height of a tree in this forest (see Fig.~\ref{fig:tprimen} for an example).

\subsection{Technical estimates}
\label{ssec:estimates}
We give an asymptotic equivalent  of the quantity $h_{n}$ defined by \eqref{eq:hn} in terms of another slowly varying function involved in the generating function of $\mu$. This alternative form will be useful for the proof of Theorem \ref{thm:heightCauchy0}. 

 Let $G_\mu(t)=\sum_{k\geq 0}t^k\mu_k$ be the probability generating function of $\mu$. Under the assumption \eqref{eq:hypmu}, we may apply Karamata's Abelian theorem \cite[Theorem 8.1.6]{BGT89} (in the notation of the latter reference, we take $n=\alpha=1$, $\beta=0$ and $f_{1}(s)=G_{\mu}(e^{-s})-1+s$) and write $G_{\mu}$ in the form
\begin{equation}
\label{eq:Gmu}
G_{\mu}(s)=s+(1-s) \ell_{\star} \left(\frac{1}{1-s}\right),
\end{equation}
 with $\ell_{\star}$ a slowly varying function. Observe that $\ell_{\star}(x) \rightarrow 0$ as $ x \rightarrow \infty$. For $x>0$ set
 \[ V(y)=\int_{1}^{y} \frac{\mathrm{d} x}{x \ell_{\star}(x)},\]
 which is slowly varying by \cite[Proposition 1.5.9a]{BGT89}. 
 
\begin{lem}
\label{lem:equiv}
If $\mu$ satisfies \eqref{eq:hypmu} and $Y$ is $\mu$-distributed then 
\[
  b_{n}    \quad \sim \quad  n \E{Y\I{Y \ge a_n}} \qquad \textrm{ and }\qquad h_{n}  \quad \sim \quad V( b_{n} ).
\]
\end{lem}

\begin{proof}
First, recall that $b_{n}=n \E{X \mathbbm{1}_{|X| >  a_{n}}}$. Since $X$ and $Y-1$ have the same law, we have 
$b_{n}=n \E{ Y\mathbbm{1}_{|Y| \geq   a_{n}+2}}-n \p(Y \geq   a_{n}+2)$.  By Karamata's Abelian theorem \cite[Theorem 8.1.6]{BGT89}, we have
\begin{equation}
\label{eq:ellstar}\E{Y \I{Y \ge n}} = \sum_{k=n}^{\infty} k \mu_k \qquad \sim \qquad \ell_{\star}(n),
\end{equation}
In particular  $\E{Y \I{Y \ge n}}$ is slowly varying, which implies $ \E{ Y\mathbbm{1}_{|Y| \geq   a_{n}+2}} \sim \E{ Y\mathbbm{1}_{|Y| \geq   a_{n}}}$.  Also, since $\p(Y \geq n) \sim L(n)/n$, we have $n \p(Y \geq   a_{n}+2) \sim n\p(Y \geq a_{n}) \rightarrow 1$. This shows that $ b_{n}  \sim  n \E{Y\I{Y \ge a_n}}$. Observe that this also shows that $b_{n} \sim n \ell_{\star}(a_{n})$.

For the second statement,  by \eqref{eq:ellstar} observe that $\ell_{\star}(x) \sim \E{Y \I{Y \ge x}}$ as $x \rightarrow \infty$. In particular, $\ell_\star(x)\to 0$ as $x\to\infty$ so $V(y)\to \infty$ as $y\to \infty$. Therefore,
\[
V(y) \quad \sim  \quad \int_1^{y} \frac{1}{ x \E{Y \I{Y \ge x}}} \mathrm{d}x\, 
\]
as $y \rightarrow \infty$. Since $V$ is slowly varying, by the first  statement of the Lemma we get that $V(b_{n}) \sim V(n \E{Y\I{Y \ge a_n}})$ and the desired result follows.
\end{proof}

\begin{rk}
\label{rk:anbn}
From the definition of $a_{n}$ we have $n L(a_{n}) \sim a_{n}$, which implies that $(a_{n})$ is regularly varying of index $1$. In addition, we have seen in the previous proof that $b_{n} \sim  n \ell_{\star}(a_{n})$, which implies that $b_{n}$ is also regularly varying of index $1$.  Finally,  $a_{n}/b_{n} \sim L(a_{n})/\ell_{\star}(a_{n}) \rightarrow 0$ because $\ell_{\star}(n) \sim \sum_{k=n}^{\infty} \frac{L(k)}{k} $  implies $L(n)=o(\ell_{\star}(n))$ by \cite[Proposition 1.5.9b]{BGT89}.
\end{rk}

Observe that the convergence of the width stated in Theorem \ref{thm:heightCauchy0} readily follows from Corollary \ref{cor:width} and Lemma \ref{lem:equiv}. For the convergence of the height, there is still some work remaining.

\subsection{The height of a forest of  Cauchy--Bienaym\'e trees}

To simplify notation, denote by $\mathcal{H}$ the height of an unconditioned $\mu$-Bienaym\'e tree. Set
\[Q_{n}=\P(\mathcal{H} \geq n).\]
Our main estimate is the following.
\begin{lem}
\label{lem:cvproba} For every $\varepsilon>0$ we have, as $n \rightarrow \infty$,
\[
 b_{n}  Q_{(1+\varepsilon)h_{n}}  \quad \longrightarrow\quad 0 \qquad \textrm{and} \qquad    b_{n}  Q_{(1-\varepsilon)h_{n}}  \quad \longrightarrow\quad \infty.
\]
\end{lem}
We claim that this immediately implies that for every $c>0$  we have 
\begin{equation}
\label{eq:hforest}
\frac{\max_{1 \leq i \leq  c  b_{n} }\mathsf{H}(\tau_{i}) }{h_{n}}  \quad \mathop{\longrightarrow}^{(\P)} \quad  1, 
\end{equation}
where  we recall that $(\tau_{i})_{i \geq 1}$ is a sequence of i.i.d.\ $\mu$-Bienaym\'e trees. 
Indeed, for any $\varepsilon>0$,
\begin{align*}
\P\left(\max_{1 \leq i \leq  c  b_{n} }\mathsf{H}(\tau_{i}) >(1+\varepsilon)h_n \right)&\leq c  b_{n}  Q_{(1+\varepsilon)h_{n}} \to 0
\intertext{by a union bound and}
\P\left(\max_{1 \leq i \leq  c  b_{n} }\mathsf{H}(\tau_{i}) <(1-\varepsilon)h_n \right)&=(1-Q_{(1-\varepsilon)h_{n}})^{c  b_{n} }\\& \leq \exp\left( - c  b_{n}  Q_{(1-\varepsilon)h_{n}}\right) \to 0
\end{align*}
since $1-x\leq e^{-x}$ for $x\ge 0$. 

We now explain how to use \eqref{eq:hforest} to prove Theorem \ref{thm:heightCauchy0}, then return to the proof of Lemma~\ref{lem:cvproba}.

\begin{proof}[Proof of Theorem \ref{thm:heightCauchy0}] We have already proved the result for the width, in Corollary~\ref{cor:width}. To prove the result for the height, by Theorem \ref{thm:dTVLoc} it is enough to check that the convergence ${\mathsf{H}(\rtree'_{n})}/{h_{n}} \rightarrow 1$ holds in probability, where we keep the notation $\rtree'_{n}$  introduced in Section \ref{ssec:forest}. For $\varepsilon \in (0,1)$ write 
\[
\p \left( \max_{1 \leq i \leq |S_{n-1}|}\mathsf{H}(\tau_{i}) \geq (1+\varepsilon) h_{n} \right) \leq \p \left( \max_{1 \leq i \leq 2 b_{n} }\mathsf{H}(\tau_{i}) \geq (1+\varepsilon) h_{n} \right)+\p \left( |S_{n-1}| \geq 2  b_{n}  \right)
\]
and
\[
\p \left( \max_{1 \leq i \leq |S_{n-1}|}\mathsf{H}(\tau_{i}) \leq (1-\varepsilon) h_{n} \right) \leq \p \left( \max_{1 \leq i \leq \frac{ b_{n} }{2}}\mathsf{H}(\tau_{i}) \leq (1-\varepsilon) h_{n} \right)+\p \left( |S_{n-1}| \leq \frac{ b_{n} }{2}  \right).
\]

By \eqref{eq:cvCauchy} we have $(S_{n-1}+b_n)/a_{n} \rightarrow 1$ in probability. Since $a_n=o(b_n)$, this implies that $|S_n/b_n| \to 1$ in probability. Since the trees $(\tau_i,i \ge 1)$ are i.i.d.\ $\mu$-Bienaym\'e trees, it then follows from \eqref{eq:hforest}  that 
\begin{equation}
\label{eq:cvproba}\frac{\max_{1 \leq i \leq |S_{n-1}|}\mathsf{H}(\tau_{i}) }{h_{n}}  \quad \mathop{\longrightarrow}^{(\P)} \quad  1.
\end{equation}
Taking into account \eqref{eq:boundH}, since $\mathsf{H}(\mathcal{F}^{n}_{\mathrm{L}})$ and $\mathsf{H}(\mathcal{F}^{n}_{\mathrm{R}})$ have the same law by symmetry and since $\p(S_{n-1}>0) \rightarrow 0$, it remains to check that for every $\varepsilon>0$
\begin{equation}
\label{eq:rightforest}
\frac{\|v'_n\|}{h_{n}}  \quad \mathop{\longrightarrow}^{(\P)} \quad 0 \qquad \textrm{and} \qquad \p \left( \frac{\max_{ |S_{n-1}| < i \leq |\underline{S}_{n-1}|}\mathsf{H}(\tau_{i})}{h_{n}} \geq 1+\varepsilon \right)  \quad \longrightarrow\quad 0.
\end{equation}

To bound $\|v'_n\|$, by Theorem 2 and Remark 18 in \cite{KR19} (recall that in the notation of \cite{KR19}, $-b_{n}$ is used instead of $b_{n}$) we have
\begin{equation}
\label{eq:exp}{\ell_{\star}( b_{n} )} \|v'_n\|  \quad \mathop{\longrightarrow}^{(d)} \quad \mathsf{Exp}(1).
\end{equation}
Moreover, by Lemma \ref{lem:equiv},
\[
h_{n} {\ell_{\star}( b_{n} )}  \sim \ell_{\star}( b_{n} )   \int_{1}^{ b_{n} } \frac{\mathrm{d} x}{x \ell_{\star}(x)} \rightarrow \infty
\]
by\cite[Proposition 1.5.9b]{BGT89}, and combined with \eqref{eq:exp} this  implies that $\|v'_n\|/h_{n} \rightarrow 0$ in probability.

Finally, to bound $\max_{ |S_{n-1}| < i \leq |\underline{S}_{n-1}|}\mathsf{H}(\tau_{i})$, first observe that by applying  Proposition \ref{prop:functional_conv_W} with $\rtree'_{  n}$ (which is licit thanks to Theorem \ref{thm:dTVLoc}) we get $(S_{n-1}-\underline{S}_{n-1})/b_{n} \rightarrow 0$ in probability as $n \rightarrow \infty$.  Thus
\begin{eqnarray*}
&&\p \left(  \frac{\max_{ |S_{n-1}| < i \leq |\underline{S}_{n-1}|}\mathsf{H}(\tau_{i})}{h_{n}}   \geq 1+\varepsilon \right) \\
&& \qquad  \qquad \leq 
\p \left(  \frac{\max_{1 \leq i \leq   b_{n} }\mathsf{H}(\tau_{i})}{h_{n}}   \geq 1+ \varepsilon \right) + \p \left(|S_{n-1}-\underline{S}_{n-1}|  >  b_{n}  \right),
\end{eqnarray*}
which tends to $0$ thanks to \eqref{eq:hforest}. 
This completes the proof.
\end{proof}

The proof of Lemma \ref{lem:cvproba} relies on estimates due to Nagaev \& Wachtel \cite{NW07}.  We first introduce some notation. 
For $s>0$ we set $\ell(s)=\ell_{\star}(1/s)$,  so that $G_{\mu}(s)=s+(1-s) \ell \left( 1-s\right)$ and 
 \begin{equation}\label{eq:ell}
 \ell(s)=\mu_0 -\sum_{k=1}^\infty \mu([k+1,\infty)) (1-s)^k.
 \end{equation}
 Indeed,
\begin{align*}
\ell(1-s)&= \frac{G_{\mu}(s)-s}{1-s} = \frac{1}{1-s}\sum_{k=0}^\infty \mu_k(s^k-s)\\
&= \frac{1}{1-s}\left(\mu_0(1-s)-\sum_{k=2}^\infty \mu_k(1-s)(s^{k-1}+\dots+s)\right)\\
&=\mu_0 -\sum_{k=1}^\infty \mu([k+1,\infty)) s^k
\end{align*}
and \eqref{eq:ell} follows.
As a consequence,  $\ell(0)\coloneqq \lim_{s\downarrow0}\ell(s)=0$, $\ell(1)<1$ and $\ell$ is increasing. A fortiori, $x \mapsto x \ell(x)$ is increasing. Finally, recalling that $Q_n = \P(\mathcal{H} \geq n)$, observe that $Q_{n+1}=1-G_\mu(1-Q_n)=Q_{n}(1-\ell(Q_{n}))$. 

We shall use two estimates due to Nagaev \& Wachtel \cite{NW07}. We have included the proofs for these estimates in our paper to ensure both self-containment and better accessibility. The proofs for the desired estimates do essentially appear in \cite{NW07}, but the estimates are not explicitly stated and their proofs are distributed throughout the paper, which can make it somewhat challenging to see the full picture.

\begin{lem}[Nagaev \& Wachtel, proof of Lemma 5 in \cite{NW07}]
\label{lem:1}
There exists a constant $C>0$ such that for every $n \geq 1$
\[\frac{1}{1-\ell(Q_{n})}\frac{\ell(Q_{n})}{\ell(Q_{n+1})} \leq  1+C \ell(Q_{n})\]
\end{lem}

\begin{proof}
From \eqref{eq:ell}, since $\mu$ is critical, for  $0 \leq s \leq 1$, we have
\[
\ell(s)=\sum_{k=1}^{\infty} \mu([k+1,\infty)) (1-(1-s)^{k}),
\]
and the inequality $k s (1-s)^{k-1} \leq 1-(1-s)^{k}$ then implies that $\ell'(s)/\ell(s) < 1/s$. If we integrate this inequality from $x$ to $y$ for $0<x<y \leq 1$ we get that $ \ln ({\ell(y)}/{\ell(x)}) < \ln(   {y}/{x})  < y/x-1$. Thus
 \[
 \frac{\ell(Q_{n})}{\ell(Q_{n+1})} \leq  \exp \left(  \frac{Q_{n}}{Q_{n+1}}-1\right)= \exp \left( \frac{1}{1-\ell(Q_{n})}-1 \right) =  \exp \left( \frac{\ell(Q_{n})}{1-\ell(Q_{n})} \right) \leq 1+C' \ell(Q_{n})
 \]
and
 \[
 (1-\ell(Q_{n}))^{-1} \frac{\ell(Q_{n})}{\ell(Q_{n+1})} \leq   (1-\ell(Q_{n}))^{-1}    \exp \left( \frac{\ell(Q_{n})}{1-\ell(Q_{n})} \right) \leq 1+C \ell(Q_{n}) 
 \]
 for certain constants $C,C'>0$, since for all $0\leq s\leq 1$ we have that $0\leq \ell(s) \leq \ell(1) <1$. 

\end{proof}

For $ y \in (0,1)$, set \[V_{\star}(y)=V(1/y)=\int_{y}^{1} \frac{\mathrm{d} x}{x \ell(x)}.\]

\begin{lem}[Nagaev \& Wachtel, Eq.~(51) in \cite{NW07}] 
\label{lem:2}
There exists a constant $C>0$ such that
we have
\[
n < V_{\star}(Q_{n}) < n - C \ln V_{\star}^{-1}(n).
\]
\end{lem}

\begin{proof}
Using the facts that $x \mapsto x \ell(x)$ is increasing and that $Q_{i}-Q_{i+1}=Q_i\ell(Q_i)$, write
\begin{equation*}\label{eq:lower_bound_W} n= \sum_{i=0}^{n-1} \frac{Q_{i}-Q_{i+1}}{Q_{i} \ell(Q_{i})}<  \sum_{i=0}^{n-1} \int_{Q_{i+1}}^{Q_{i}} \frac{\mathrm{d}x}{x \ell(x)}=\int_{Q_{n}}^{1} \frac{\mathrm{dx}}{x \ell(x)}=V_{\star}(Q_n),
\end{equation*}
which gives the lower bound for $V_{\star}(Q_n)$ in the lemma.
For the upper bound, we see that 
\[
V_{\star}(Q_n)=\sum_{i=0}^{n-1} \int_{Q_{i+1}}^{Q_{i}} \frac{\mathrm{d}x}{x \ell(x)} < \sum_{i=0}^{n-1} \frac{Q_{i}-Q_{i+1}}{Q_{i+1} \ell(Q_{i+1})} = \sum_{i=0}^{n-1} \frac{Q_{i}\ell(Q_i)}{Q_{i+1} \ell(Q_{i+1})}= \sum_{i=0}^{n-1} \frac{1}{1-\ell(Q_{i})} \frac{\ell(Q_i)}{\ell(Q_{i+1})}. 
 \]
 Thus, by Lemma \ref{lem:1} we have
 \[V_{\star}(Q_{n}) < n +C \sum_{i=0}^{n-1} \ell(Q_{i}).\]
 Next, since the lower bound for $V_{\star}(Q_i)$ gives that $i<V_{\star}(Q_{i})$, using the fact that $V_{\star}$ is decreasing and $\ell$ is increasing, write
 \[ \sum_{i=0}^{n-1} \ell(Q_{i}) <  \sum_{i=0}^{n} \ell (V_{\star}^{-1}(i)) < \int_{0} ^{n} \ell(V_{\star}^{-1}(x)) \mathrm{d}x.\]
 The change of variable $x=V_{\star}(u)$ with $V_{\star}'(u)= - \frac{1}{u \ell (u)}$ gives 
 \[ \int_{0} ^{n} \ell(V_{\star}^{-1}(x)) \mathrm{d}x= \int_{V_{\star}^{-1}(n)}^{1} \ell(u)  \cdot  \frac{1}{u \ell (u)} \mathrm{d}u= - \ln V_{\star}^{-1}(n),\]
 and the desired result follows.
\end{proof}

\begin{lem}
\label{lem:3}We have $\ln(V_{\star}^{-1}(n))/n \rightarrow 0$ as $ n \rightarrow \infty$.
\end{lem}

\begin{proof}

We have $0<V_{\star}^{-1}(n)<1$ so that $\ln(V_{\star}^{-1}(n))<0$. We also have that  $(\ln V_{\star}^{-1}(x))'=-\ell(V_{\star}^{-1}(x))$. Since $V_{\star}^{-1}(x) \rightarrow 0$ as $x \rightarrow \infty$ and since $\ell(0)=0$, we see that for every $\varepsilon>0$ we can find $M>0$ such that $-\ell(V_{\star}^{-1}(x)) \geq -\varepsilon$ for $x \geq M$. Integrating this inequality, it follows that there exists a constant $C>0$ such that for $n \geq M$
\[
C -\varepsilon n \leq \ln V_{\star}^{-1}(n) \leq 0
\]
The result follows.
\end{proof}

\begin{proof}[Proof of Lemma \ref{lem:cvproba}]
By \cite[Theorem 2.4.7 (i) and Eq.~(2.4.3)]{BGT89}, since $V_{\star}$ is slowly varying and decreasing, $V_{\star}^{-1}$ is rapidly varying and decreasing, meaning that for every $\varepsilon \in (0,1)$ we have $V_{\star}^{-1}((1+\varepsilon) n)/V_{\star}^{-1}(n) \rightarrow 0$. Also, if $a_{n}$ is a sequence such that $a_{n}/n \rightarrow 1-\varepsilon$, we have $V_{\star}^{-1}( a_{n })/V_{\star}^{-1}(n) \rightarrow \infty$.

To simplify notation, set $h'_{n}=V( b_{n} )$. By definition, $h'_{n}=V_{\star}(1/b_n)$, so $b_n=1/V_{\star}^{-1}(h'_{n})$. By Lemma \ref{lem:equiv}, we have $h_n \sim h'_{n}$. Thus for $n$ sufficiently large we have $ (1+\varepsilon)  h_{n} \leq (1+2\varepsilon) h'_{n}$ and$ (1-\varepsilon)  h_{n} \leq (1-2\varepsilon) h'_{n}$.

Then, by Lemma \ref{lem:2}, 
\[
 b_{n}  Q_{(1+\varepsilon)h_{n}} \leq b_{n}  Q_{(1+2\varepsilon)h'_{n}}   \frac{V_{\star}^{-1}((1+2\varepsilon)h'_{n})}{V_{\star}^{-1}(h'_{n})}  \quad \longrightarrow\quad 0.
\]

For the other convergence, using Lemma \ref{lem:2}, similarly write
\[
 b_{n}  Q_{(1-\varepsilon)h_{n}}  \geq   b_{n}  Q_{(1-2\varepsilon)h'_{n}}    \geq  \frac{V_{\star}^{-1}((1-2\varepsilon)h_{n} - C \ln V_{\star}^{-1}((1-2\varepsilon)h_{n}))}{V_{\star}^{-1}(h'_{n})} \quad \longrightarrow\quad \infty,
\]
where we use that \[\frac{ (1-2\varepsilon)h'_{n} - C \ln V_{\star}^{-1}((1-2\varepsilon)h'_{n})}{h_n} \to 1-2\varepsilon\] by Lemma \ref{lem:3}.
This completes the proof.
\end{proof}

\subsection{Examples}
\label{ssec:examples}

Here we give the details concerning the examples mentioned in the end of the Introduction. Starting from $  \mu_n=  {L(n)}/{n^{2}}$, we compute the asymptotic equivalents of $\ell_{\star}$ and $V$ using \eqref{eq:ellstar} and the definition of $V$:
\begin{equation}
\label{eq:integrals}\ell_{\star}(n)  \quad \sim \quad  \int_{n}^{\infty} \frac{L(x)}{x} \mathrm{d}x, \qquad V(y)=\int_{1}^{y} \frac{\mathrm{d} x}{x \ell_{\star}(x)}.
\end{equation}
We extract an asymptotic equivalent for $b_{n}$ from the implicit asymptotic equivalent $ b_{n} \sim  n \ell_{\star}( b_{n} )$, which comes from \cite[Lemma 4.3]{Ber19}. (We have already seen that $ b_{n} \sim  n \ell_{\star}( a_{n} )$ but here it is more convenient to use that $ b_{n} \sim  n \ell_{\star}( b_{n} )$). This allows to compute an asymptotic equivalent for $V( b_{n} )$, which by Lemma \ref{lem:equiv} is an asymptotic equivalent of $h_{n}$. In the following examples the  integrals appearing in \eqref{eq:integrals} can actually be computed explicitly.
 
 \begin{enumerate}
 \item[--]  $L(n) \sim \beta \ln(n)^{-1-\beta} $ with $\beta>0$, then $\ell_{\star}(s) \sim \ln(s)^{-\beta}$ and $V(x) \sim  \frac{1}{1+\beta} \ln(x)^{1+\beta}$. Then $ b_{n}  \sim n/\ln( b_{n} )^{\beta}$, which implies that $ b_{n}  \sim n/\ln(n)^{\beta}$ and thus
 \[
 h_{n}  \quad \sim \quad  V( b_{n} )   \quad \sim \quad   \frac{1}{1+\beta} \ln( b_{n} )^{1+\beta}  \quad \sim \quad \frac{1}{1+\beta} \ln(n)^{1+\beta}.
 \]

 \item[--] Set $\ln_{(1)}(x)=\ln(x)$ and for every $k \geq 1$ define recursively $\ln_{(k+1)} (x) = \ln(\ln_{(k)}(x))$. For
 \[
 L(n) \quad \sim \quad  \frac{1}{(\ln_{(k)}(n))^{2}} \prod_{i=1}^{k-1} \frac{1}{\ln_{(i)}(n) } \]
 with $k \geq 2$, we have $\ell_{\star}(s)=\ln_{(k)}(s)^{-1}$ and $V(x) \sim \ln(x) \ln_{(k)}(x)$. Then $ b_{n}  \sim n \ln_{(k)}( b_{n} )^{-1}$ which implies that $ b_{n}  \sim n \ln_{(k)}(n )^{-1}$ and thus
 \[
 h_{n}  \quad \sim \quad V( b_{n} )   \quad \sim \quad   \ln( b_{n} ) \ln_{(k)}( b_{n} )  \quad \sim \quad  \ln(n) \ln_{(k)}(n).
 \]
  \item[--] $L(n) \sim  \frac{1-\beta}{\beta} \ln(n)^{-\beta} e^{-\ln(n)^{\beta}} $ with $\beta \in (0,1)$, then $\ell_{\star}(s) \sim \frac{1}{\beta} \ln(s)^{1-\beta} e^{-\ln(s)^{\beta}} $ and $V(x) \sim e^{\ln(x)^{\beta}}$. Then, $ b_{n}  \sim \frac{n}{\beta} \ln( b_{n} )^{1-\beta} e^{-\ln( b_{n} )^{\beta}} $. Thus, taking the $\beta$-th power of the logarithm of this equivalent and iterating yields
 \[
 h_{n}  \quad \sim \quad V( b_{n} )   \quad \sim \quad   e^{\ln( b_{n} )^{\beta}}  \quad \sim \quad \exp \left( \sum_{i=0}^{k}  P_{i}(\beta) \ln(n)^{\beta-i(1-\beta)} \right),
 \]
 where $k \geq 0$ is the smallest integer such that $ \frac{k}{k+1} \leq \beta < \frac{k+1}{k+2}$ and where $P_{i}(\beta)$ is a polynomial in $\beta$ of degree $i$  which can be computed by induction (e.g.~$P_{0}=1$, $P_{1}(\beta)=-\beta$).
 Observe that in contrast with the two previous examples, here, in general, it is neither true that $h_{n} \sim V(n)$ nor that $h_{n} \sim V(n \ell_{\star}(n))$.
 \end{enumerate}
 
Finally, we note that Theorem \ref{thm:heightCauchy0} is consistent with Theorem \ref{thm:height_critical0}. Indeed, we have $b_{n}=o(n)$ since $ b_{n} \sim  n \ell_{\star}(a_{n})$ and so to see that
 \begin{equation}
 \label{eq:hnbn} \frac{h_{n}}{\ln(n)}  \quad \longrightarrow\quad \infty,
 \end{equation}
 take $\varepsilon>0$ and $M>0$ such that $\ell_{\star}(x) \leq \varepsilon$ for $x \geq M$. Then
\[ \frac{1}{\ln(n)} \cdot \int_{1}^{ b_{n} }  \frac{ \mathrm{d}x }{x \ell_{\star}(x)} \geq \frac{1}{\ln(n)} \cdot \int_{M}^{ b_{n} }  \frac{ \mathrm{d}x }{x \varepsilon} = \frac{\ln( b_{n} )-\ln(M)}{\ln(n)} \cdot \frac{1}{\varepsilon}.\]
But $ b_{n} /n$ is slowly varying, and if $\Lambda$ is a slowly varying function we have $\ln(\Lambda(n))/\ln(n) \rightarrow 0$ as $n \rightarrow \infty$ (see \cite[Proposition 1.3.6(i)]{BGT89}), so that $\ln(b_{n}) \sim \ln(n)$. As a consequence
\[
\liminf_{n \rightarrow \infty}  \frac{1}{\ln(n)} \cdot \int_{1}^{ b_{n} }  \frac{ \mathrm{d}x }{x \ell_{\star}(x)} \geq    \frac{1}{\varepsilon}
\]
and \eqref{eq:hnbn} follows. 

We finish the section by proving Corollary~\ref{cor:hw}, which stated that $\mathsf{H}(\rtree_{n})\mathsf{Width}(\rtree_{n}) = O_{\p}({n \ln(n)})$.
\begin{proof}[Proof of Corollary~\ref{cor:hw}]
By Theorem \ref{thm:heightCauchy0} and Lemma \ref{lem:equiv}, we have the following convergence in probability:
\[
\frac{\mathsf{H}(\rtree_{n}) \cdot \mathsf{Width}(\rtree_{n}) }{ b_{n}  V( b_{n} )} \quad \xrightarrow[n\to\infty]{} \quad 1.
\]
We thus have to establish that $ b_{n}  V( b_{n} )  = O(n \ln(n))$.
To this end, using the definition of $V$ write
\[
 b_{n}  V( b_{n} )= b_{n}  \int_{1}^{ b_{n} } \frac{\mathrm{d} x}{x \ell_{\star}(x)}=  b_{n}  \int_{1/ b_{n} }^{1} \frac{ \mathrm{d}u}{u \ell_{\star}( b_{n}  u)}.
\]
As seen in the beginning of Sec.~\ref{ssec:examples}, we have  $  b_{n}  \sim  n \ell_{\star}( b_{n} )$, so that
\[
 b_{n}  V( b_{n} )  \quad \sim \quad  n \int_{1/ b_{n} }^{1} \frac{1}{u} \frac{\ell_{\star}( b_{n} )}{ \ell_{\star}( b_{n}  u)}  \mathrm{d}u.
\]
But there exists a constant $C>0$ such that for every $ n \geq 1$ and for every $1/ b_{n}  \leq u \leq 1$ we have  $\ell_{\star}( b_{n} )/\ell_{\star}( b_{n} u) \leq C$ (this follows e.g.~from the representation theorem for slowly varying functions \cite[Theorem 1.3.1]{BGT89}). Thus $ b_{n}  V( b_{n} )   \leq C n \ln( b_{n} )$, 
and the desired result follows from the fact that $\ln( b_{n} ) \sim \ln(n)$ as seen above.
\end{proof}

Observe that in the first two family of examples above $b_{n} V(b_{n})$ is of order $n \ln(n)$, while in the third family of examples we have
\[
  b_{n}  V( b_{n} )   \quad \sim \quad \frac{n}{\beta} \ln(n)^{1-\beta} = o(n \ln(n)).
 \]
 
\section{\bf The height {and width} of large conditioned critical Bienaym\'e trees}\label{sec:height_critical}
In this section we will prove Theorem~\ref{thm:height_critical0}.

\subsection{Degree sequences of Bienaym\'e trees}\label{ssec:deg_seq_bien}
We first establish some properties concerning degree sequences of Bienaym\'e trees {that will be instrumental in our proof of Theorem \ref{thm:height_critical0}.}

We shall use the following fact about unconditioned samples from $\mu$, which can be found in \cite[Lemmas 14.3 and~14.4]{Jan12}. We provide a proof for completeness and since our argument is {more probabilistic in nature and also} somewhat shorter than that in \cite{Jan12}.
\begin{lem}\label{lem:probtree}
For $(X_i,i \ge 1)$ i.i.d.\ samples with law $\P(X_1=i)=\mu_{i+1}$, we have that 
\[\P\left(\sum_{i=1}^n X_i=-1\right)=e^{-o(n)}\]
as $n\to \infty$ over all $n$ for which $\P\left(\sum_{i=1}^n X_i=-1\right)>0$.
\end{lem}
This lemma implies that every event that is exponentially unlikely for $\rtree$ is also exponentially likely for $\rtree_n$. Indeed, by the ballot theorem (or equivalently by Proposition~\ref{prop:GWRW} and the fact that the Vervaat transform is an $n$-to-$1$ map), \[\P_\mu(|\rtree|=n)=\frac{1}{n}\P\left(\sum_{i=1}^n X_i=-1\right),\]
which implies that for any set $\cA$ of plane trees,
\[ \P_\mu(\rtree_{n} \in \cA)\leq \frac{\P(\rtree \in \cA)}{\P_\mu(|\rtree|=n)}\leq ne^{o(n)} \P(\rtree \in \cA).\]

\begin{proof}[Proof of Lemma~\ref{lem:probtree}]

In the proof, we only consider $n$ for which $\P\left(\sum_{i=1}^n X_i=-1\right)>0$.

Now, let $r=\operatorname{gcd}(i> 0 :\mu_{i}>0)$,   so that the support of $\sum_{i=1}^n X_i$ is contained in $-n+r\Z$. Then, there is $k\in \N$ such that $r=\operatorname{gcd}(0\leq i\leq  k:\mu_i>0)$. Set $\mathcal{S}=\{0\leq i\leq  k:\mu_i>0\}$ and let $\mathcal{S}'=\{i>k :\mu_i>0\}$. If $\mathcal{S}'=\emptyset$, then we are done, because this in particular implies that $\mu$ has finite variance, so by the local central limit theorem  (see e.g. \cite[Theorem 1 in Sec.~VII]{Petrov}), 
\[ \P\left(\sum_{i=1}^n X_i=-1\right)=\Theta(n^{-1/2})\]
as $n\to \infty$ over all $n$ for which $\P\left(\sum_{i=1}^n X_i=-1\right)>0$.
Therefore, we may assume that $\cS'\neq \emptyset$. 

Set $p=\sum_{i\in \cS} \mu_{i}$. Then for $N_{n}\coloneqq\#\{1\leq i \leq n: X_i+1\in \cS\}$, we see that $N_{n}$ is distributed as a $\operatorname{Binomial}(n,p)$ random variable. Moreover, letting $(A_i,i \ge 1)$ be i.i.d.\ random variables distributed as $X_1$ conditional on $X_1+1\in \cS$ and $(A'_i,i \ge 1)$ be i.i.d.\ random variables distributed as $X_1$ conditional on $X_1+1 \in \cS'$, both sequences independent of one another and of $N_n$, then 
\[
\sum_{i=1}^{N_{n}}A_i+\sum_{i=1}^{n-N_{n}}A'_i \quad \overset{d}{=}  \quad \sum_{i=1}^n X_i
\]
Let $c=\E{A_1}<0$ and $c'=\E{A'_1}>0$ be the means of $A_1$ and $A'_1$ respectively, which satisfy that $pc+(1-p)c'=0$ by the criticality condition.

Fix $\varepsilon>0$. We are done if we show that 
\[\P\left(\sum_{i=1}^n X_i=-1\right)\geq \tfrac{1}{4} e^{-\varepsilon n}\]
for all $n$ large enough. To this end, we first observe that there is $M>0$ such that $\P(|N_{n}-np|<Mn^{1/2})>1/2$ for all $n$. Furthermore, since $\sum_{i=1}^mA_i$ is the sum of i.i.d\ random variables with mean $c$ and finite support, there exists $\delta>0$ such that for all $m_n$ with $|m_n-np|<Mn^{1/2}$ and for all $k_n\in -m_n+ r\Z$ with $|k_n-npc|<\delta n$, for all $n$ large enough,
\begin{equation}\label{eq:A}\P\left(\sum_{i=1}^{m_n}A_i=k_n\right)\geq e^{-\varepsilon n}.\end{equation}

We now check that for $n$ large enough and for all $m_n$ with $|m_n-np|<Mn^{1/2}$
\begin{equation}
\label{eq:Aprime}
 \P \left( \left|\sum_{i=1}^{n-m_n}A'_i +npc +1 \right| \leq \delta n \right) \geq  \frac{1}{2}.
\end{equation}
Then, with probability at least $\tfrac14$, we get that $m_n=N_n$ and $k_n=-1-\sum_{i=1}^{n-m_n}A'_i$ satisfy the conditions of \eqref{eq:A}. (Observe that since $\P(\sum_{i=1}^n X_i=-1)>0$ we have that $n-1\in r\Z$, and  $\sum_{i=1}^{n-m_n}A'_i\in -n+m_n+r\Z$ by definition of $r$, so that $k_n\in n-1-m_n+r\Z=-m_n+r\Z$.) This implies that with probability at least $\tfrac14 e^{-\varepsilon n}$,
it holds that $\sum_{i=1}^{N_{n}}A_i+\sum_{i=1}^{n-N_{n}}A'_i=-1$.

Finally, to establish \eqref{eq:Aprime}, observe that by the law of large numbers, for $|m_n-np|<Mn^{1/2}$ it holds that $|\sum_{i=1}^{n-m_n}A'_i-n(1-p)c'+1|<\delta n$ with probability at least $1/2$ for $n$ large enough, so the fact that $pc+(1-p)c'=0$ implies the statement.
\end{proof}

To establish the bound concerning $\mathsf{H}(\rtree_{n})$, we will also use the following lemma, which states that in $\rtree_n$, all but $o(n)$ of the total degree is at bounded degree nodes with high probability. This in particular implies that the largest degree is $o(n)$ in probability. 

\begin{lem}\label{lem:nomasshighdegrees}
For every $\beta,\varepsilon \in (0,1)$, for $D_1,\dots, D_n$ the degrees in $\rtree_n$ there exists $K>0$ such that for $n$ large enough,
\[\P\left(\sum_{i=1}^n D_i\I{D_i> K}\leq \beta n \right)>1-\varepsilon.\]
\end{lem}

\begin{proof}
Fix $\beta \in (0,1)$. Let $(X_i,i \ge 1)$ be i.i.d.\ samples with law $\P(X_1=i)=\mu_{i+1}$. By Remark \ref{rk:degrees},  the multiset $\{D_1,\ldots,D_n\}$ of degrees in $\rtree_n$ has the same law as the conditional law of the multiset $\{X_i+1,1 \le i \le n\}$ given that $\sum_{i=1}^n (X_i+1)=n-1$. 

We will first consider $(X_1+1,\dots,X_n+1)$ before conditioning: we will show that there exist $\delta>0$ and $K>0$ so that for all $n$ large enough, 
\begin{equation}\label{eq:meanlargeiidcase}\P\left(\sum_{i=1}^n (X_i+1) \I{X_i\leq  K-1}\leq (1-\beta) n \right) <e^{-\delta n}.\end{equation}
Since $\mu$ is critical, there is $K\in \N$ such that $\sum_{k=1}^K k\mu_k>1-\beta/2$. For such $K$ we have 
\[\left\{ \sum_{i=1}^n (X_i+1) \I{X_i\leq  K-1}\leq (1-\beta) n\right \} \subseteq \bigcup_{\{k\in [K]:\mu_k>0\}} \left\{  \sum_{i=1}^n \I{X_i= k-1} \leq 
\frac{1-\beta}{1-\beta/2} \mu_k n \right\}.\]
{ Indeed, if $\sum_{i=1}^n \I{X_i= k-1} > (1-\beta)\mu_k n/({1-\beta/2})$ for all $k\in [K]$ for which $\mu_k>0$, then 
\begin{align*}
\sum_{i=1}^n (X_i+1) \I{X_i\leq  K-1}&=\sum_{k=1}^K k\sum_{i=1}^n\I{X_i= k-1}\\
&\ge \sum_{k=1}^K k \frac{1-\beta}{1-\beta/2} \mu_k n\\
&> (1-\beta )n .
\end{align*}
}
But, for each $k$, $ \sum_{i=1}^n \I{X_i= k-1}$ is distributed as a $\operatorname{Binomial}(n,\mu_k)$ random variable, so  for each $k\leq K$ for which $\mu_k>0$, there is $\delta_k>0$ such that 
\[\P\left(\sum_{i=1}^n \I{X_i= k-1} \leq 
\frac{1-\beta}{1-\beta/2} \mu_k n\right) <e^{-\delta_kn} \]
for all $n$ large enough. Then a union bound implies \eqref{eq:meanlargeiidcase} with, e.g., $\delta=\tfrac{1}{2}\min_{\{k\in [K]:\mu_k>0\}}\{\delta_k\}$.

 We now show how the statement follows from \eqref{eq:meanlargeiidcase} and Lemma~\ref{lem:probtree}. By Lemma~\ref{lem:probtree}, for all $n$ large enough, $\P(\sum_{i=1}^n (X_i+1)=n-1)>e^{-\delta n/2} $, so 
\begin{align*}
\P\left(\sum_{i=1}^n D_i\I{D_i> K}> \beta n \right)&= \P\left(\sum_{i=1}^n (X_i+1) \I{X_i+1\leq  K}<n-1-\beta n \mid \sum_{i=1}^n (X_i+1)=n-1 \right)\\
&\leq \frac{\P\left(\sum_{i=1}^n (X_i+1) \I{X_i\leq  K-1}\leq (1-\beta) n \right) }{\P(\sum_{i=1}^n X_i=-1)}<e^{-\delta n/2},\end{align*}
which implies the desired result.

\end{proof}

{\subsection{Critical trees are not too fat}
Here we establish the width bound in Theorem \ref{thm:height_critical0}}. We may assume throughout that $\mu_0+\mu_1<1$ since if $\mu_0+\mu_1=1$ then the assertions of the theorem are immediate. We start with the bound concerning $\mathsf{Width}(\rtree_{n})$.
Recall that $S_m=X_1+\ldots+X_m$ for $m \ge 1$, where $(X_i,i \ge 1)$ are i.i.d.\ with $\p(X_1=k)=\mu_{k+1}$ for $k \ge -1$. 
\begin{prop}\label{prop:width_bound}
For all $\eps > 0$ there exists $\delta > 0$ such that for all $n$ sufficiently large,
\[
\Cprob{\max_{0 \le m \le n} |S_m| \ge \eps n}{S_n=-1} \le e^{-\delta n}.
\]
\end{prop}
\begin{proof}
If $\mu$ has bounded support then by a Chernoff bound, for all $\eps >0$ there is $\delta > 0$ such that 
$\p(\max_{0 \le m \le n} |S_m| \ge \eps n) \le e^{-2\delta n}$ for all $n$ sufficiently large. In this case, by the local central limit theorem we also have that $\p(S_n=-1)=\Theta(n^{-1/2})$ as $n \to \infty$ (along values $n$ for which $\p(S_n=-1)\ne 0$). Thus, 
\[
\Cprob{\max_{0 \le m \le n} |S_m| \ge \eps n}{S_n=-1}=O(\sqrt{n}e^{-2\delta n})=o(e^{-\delta n})\, ,
\]
which implies the result in this case. 
We may thus assume that $\mu$ has unbounded support. 

We will use below that $\p(S_n = -1) = e^{-o(n)}$ as $n \to \infty$ along values with $\p(S_n=-1)>0$; see Lemma~\ref{lem:probtree}. 

For $K \in \N$ it is convenient to write 
$X_i^{\le K}=X_i \I{X_i \le K}$ and $X_i^{>K}=X_i-X_i^{\le K}$, and also to write $S_m^{\le K}=\sum_{i=1}^m X_i^{\le K}$ and $S_m^{>K}=S_m-S_m^{\le K}$. 

Fix $\eps > 0$. Since the $X_i$ are centred, we may choose $K=K(\eps)$ large enough that $\E{X_i\I{X_i \le K}} >-\eps/2$. Then 
$\E{\sum_{i=1}^n X_i^{\le K}} > - \varepsilon n/2$. Since the random variables $X_i^{\le K}$ are bounded, by a Chernoff bound there exists $\delta=\delta(\eps)>0$ such that 
\[
\P(S_n^{\le K} < -\eps n) < \delta^{-1}e^{-2\delta n}
\]
for all $n\ge 1$. 
Also, the random variables $X_i^{\le K}$ are bounded and (since the support of $\mu$ is unbounded) have strictly negative mean. Thus there is $c > 0$ such that for all $m \ge 1$, 
\[
\P(S_m^{\le K} \le 0) \ge c.
\]
By the Markov property this implies that for $0 \le m \le n$, 
\begin{align*}
\probC{S_n^{\le K} \le -\eps n}{S_m^{\le K} \le -\eps n}
 & 
 \ge \p(S_{n-m}^{\le K} \le 0) \ge c\, ,
\end{align*}
so for $m \in [n]$ we have
\[
\p(S_m^{\le K} \le -\eps n) \le (c\delta)^{-1} e^{-2\delta n}\, .
\]
Again using the Markov property, this implies that 
\begin{align*}
\P(S_m \ge \eps n,S_n=-1)
& \le \P(S_{n-m} \le -\eps n) 
 \le \P(S_{n-m}^{\le K} \le -\eps n)
 \le (c\delta)^{-1}e^{-2\delta n}
\end{align*}
Since $\p(S_n = -1) = e^{-o(n)}$, the above bound yields that 
\[
\probC{S_m \ge \eps n}{S_n=-1} \le (c\delta)^{-1}e^{-2\delta n} \cdot e^{o(n)}\,
\]
and a union bound over $m \in [n]$ then gives that 
\[
\probC{\max_{0 \le m \le n} S_m \ge \eps n}{S_n=-1} 
\le e^{-\delta n}/2
\]
for all $n$ sufficiently large. 

Bounding $\min_{0 \le m \le n} S_m$ is similar: for $m \in [n]$ we have we have 
\begin{align*}
\p(S_m \le -\eps n,S_n=-1)
& \le \p(S_m \le -\eps n) \le 
\p(S_m^{\le K} \le -\eps n)
\le (cd)^{-1} e^{-2\delta n}\, .
\end{align*}
Just as above, dividing through by $\p(S_n=-1)$, using that $\p(S_n = -1) = e^{-o(n)}$, and taking a union bound gives that 
\[
\probC{\min_{0 \le m \le n} S_m \le -\eps n}{S_n=-1} 
\le e^{-\delta n}/2 
\]
for $n$ large; the result follows. 
\end{proof}
The width bound of Theorem~\ref{thm:height_critical0} now follows straightforwardly.
\begin{cor}\label{cor:width_bound}
For any critical offspring distribution $\mu$, it holds that $\mathsf{Width}(\rtree_n)=o_{\p}(n)$. 
\end{cor}
\begin{proof}
Let $\zeta=\inf(i \ge 1: S_i=-1)$. 
By \eqref{eq:condition_identity}, for any discrete excursion $\mathbf{w}=(w_0,w_1,\ldots,w_n)$, 
\[
\p((\W^{\mathsf{bfs}}_{i}( \rtree_n),0 \le i \le n)=\mathbf{w})
=
\Cprob{\mathcal{V}(S_i,0 \le i \le n)=\mathbf{w}}{S_n=-1}
\]
By \eqref{eq:width} we also have 
\[ 
\textsf{Width}(\rtree_n) \leq \max (\W^{\mathsf{bfs}}_m(\rtree_n),0 \le m \le n). 
\]
Finally, note that for any discrete bridge $\mathbf{s}=(s_0,s_1,\ldots,s_n)$, it holds that 
\[
\max(\mathcal{V}(\mathbf{s})_m,0 \le m \le n) = \max(s_m,0 \le m \le n)-\min(s_m,0 \le m \le n).\]
Proposition~\ref{prop:width_bound} now implies that for any $\eps > 0$, there is $\delta > 0$ such that 
\begin{align*}
\p(\textsf{Width}(\rtree_n) \ge 2\eps n)
& \le \p(\max (\W^{\mathsf{bfs}}_m(\rtree_n),0 \le m \le n) \ge 2\eps n) \\
& \le \Cprob{\max_{0 \le m \le n} S_m-\min_{0 \le m \le n} S_m\ge 2\eps n}{S_n=-1} \\
& \le \Cprob{\max_{0 \le m \le n} |S_m| \ge\eps n}{S_n=-1} \le e^{-\delta n}. \qedhere
\end{align*}
\end{proof}

\subsection{Bijection between rooted trees and sequences}
Our {proof for the height bound in Theorem \ref{thm:height_critical0}} is in part based on a bijection between rooted trees on $[n]$ and sequences in $[n]^{n-1}$ introduced by Foata and Fuchs \cite{FoataFuchs}, which we recall here.  (See also the recent note by the first and second author, Blanc-Renaudie, Maazoun and Martin on probabilistic applications of the bijection \cite{cayley_us} and the paper by the first and second author in which we use the bijection to obtain upper bounds on the height of random trees \cite{AD22}.) 

For a rooted tree $\tree$ and a vertex $v$ of $\tree$, the {\em degree} $d_\tree(v)$ of $v$ is the number of children of $v$ in $\tree$. A {\em leaf} of $\tree$ is a vertex of $\tree$ with degree zero. 

A {\em degree sequence} is a sequence of non-negative integers $\dseq=(d_1,\ldots,d_n)$ with $\sum_{i \in [n]} d_i = n-1$. We write $\cT_{\dseq}$ for the set of finite rooted labeled trees $\tree$ with vertex set labeled by $[n]$ and such that for each $i \in [n]$, the vertex with label $i$ has degree $d_i$. 
Also write $\cT(n)$ for the set of all finite rooted labeled trees with vertex set labeled by $[n]$. For $\tree \in \cT(n)$, we write $d_\tree(i)$ for the degree of the vertex with label $i$ and  say that $(d_\tree(1),\ldots,d_\tree(n))$ is the degree sequence of $\tree$. 

We use the version of the bijection introduced in \cite[Section 3]{cayley_us}, specialized to trees with a fixed degree sequence and we closely follow the presentation of the bijection in \cite{AD22}.

The following proposition from \cite{AD22} is the tool that allows us to transfer results from the setting of random trees with given degree sequences to that of Bienaym\'e trees. Recall that $\rtree_n$ denotes a $\mu$-Bienaym\'e tree conditioned to have size $n$.

\begin{prop}[Proposition 8 in \cite{AD22}]\label{prop:planetrees}
Conditionally given $\rtree_{n}$, let $\hat{\rtree}_{n}$ be the random tree in $\cT(n)$ obtained as follows: 
label the vertices of $\rtree_{n}$ by a uniformly random permutation of $[n]$, then forget about the plane structure. 
For $i \in [n]$ let $D_i$ be the degree of vertex $i$ in $\hat{\rtree}_{n}$. Then for any degree sequence $\dseq=(d_1,\ldots,d_n)$, conditionally given that $(D_1,\ldots,D_n)=(d_1,\ldots,d_n)$, the tree $\hat{\rtree}_{n}$ is a uniform element in $\cT_{\dseq}$. 
\end{prop}

We will now explain the bijection. It is convenient to focus on degree sequences with a particular form. Given a degree sequence $\dseq=(d_1,\ldots,d_n)$, define another degree sequence $\dseq'=(d_1',\ldots,d_n')$ as follows. 
Let $m$ be the number of non-zero entries of $\dseq$; necessarily $1 \le m \le n-1$. List the non-zero entries of $\dseq$ in order of appearance as $d_1',\ldots,d_m'$, and then set $d_{m+1}'=\ldots=d_n'=0$. So, for example, if $\dseq=(1,0,3,0,0,2,0)$ then $\dseq'=(1,3,2,0,0,0,0)$. Say that the degree sequence $\dseq'$ is {\em compressed}. (So a degree sequence is compressed if all of its non-zero entries appear before all of its zero entries.) 

There is a natural bijection between $\cT_{\dseq}$ and $\cT_{\dseq'}$: from a tree $\tree\in \cT_{\dseq}$, construct a tree $\tree' \in \cT_{\dseq'}$ by relabeling the non-leaf vertices of $\tree$ as $1,\ldots,m$ and the leaves of $\tree$ as $m+1,\ldots n$, in both cases in increasing order of their original labels. 
Using this bijection provides a coupling $(\rtree,\rtree')$ of uniformly random elements of $\cT_{\dseq}$ and $\cT_{\dseq'}$, respectively, such that $\rtree$ and $\rtree'$ have the same height. It follows that any bound on the height of a uniformly random tree in $\cT_{\dseq'}$ applies {\em verbatim} to the height of a uniformly random tree in $\cT_{\dseq}$. 

Now, let $\dseq=(d_1,\dots,d_n)$ be a compressed degree sequence, so there is $1 \le m \le n-1$ such that $d_i=0$ if and only if $i>m$.  Write $n_0=n_0(\dseq)=n-m$ for the number of leaves in a tree with degree sequence $\dseq$.  Then define
\begin{align*}
\cS_{\dseq}&\coloneqq\left\{(v_1,\dots, v_{n - 1}):|\{k:v_k=i\}|=d_i\text{ for all }i\in [n]\right\}\, .
\end{align*}
For example, if $\dseq=(1,3,2,0,0,0,0)$ then $\cS_{\dseq}$ is the set of all permutations of the vector $(1,2,2,2,3,3)$, so has size ${6 \choose 1,3,2} = 60$. 

The following bijection between $\cS_{\dseq}$ and $\cT_{\dseq}$ appears in \cite[Section 3]{cayley_us}. See also Figure~\ref{fig:bij}. For $\sequence{v}=(v_1,\dots, v_{n-1}) \in \cS_{\dseq}$, we say that $j\in \{2,\dots, n-1\}$ is the location of a repeat of $\sequence{v}$ if 
$v_j=v_i$ for some $i<j$. 
\begin{tcolorbox}[title=Bijection $\tree$ between $\cS_{\dseq}$ and $\cT_{\dseq}$.]
	\begin{itemize}
		\item Let $j(0)=1$, let $j(1)<j(2)<\dots<j(n_0-1)$ be the locations of the repeats of the sequence $\sequence{v}$, and let $j(n_0)=m+n_0=n$.
		\item For $i=1,\dots, n_0$, let $P_i$ be the path
		$(v_{j(i-1)}, \dots, v_{j(i)-1}, m+i)$.
		\item Let $\tree(\sequence{v}) \in \cT_{\dseq}$ have root $v_1$ and edge set given by the union of the edge sets of the paths
		$P_1, P_2, \dots, P_{n_0}$.
	\end{itemize}
\end{tcolorbox}
\begin{figure}
\centering
\includegraphics{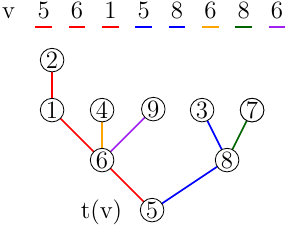}
\caption{\label{fig:bij}A sequence $\sequence{v}=(v_1,\ldots,v_8)$ and the tree $\tree(\sequence{v})$ that it encodes. For $i=1,\dots, 5$, the colour of the edges on the path $P_i$ in $\tree(\sequence{v})$ equals the colour underlining 
entries $v_{j(i-1)},\ldots,v_{j(i)-1}$ of $\sequence{v}$. 
Figure adapted from \cite{cayley_us}.}\end{figure}

The inverse of the bijection works as follows. Fix a tree $\tree \in \cT_{\dseq}$. Let $S_0 = \{r(\tree)\}$ consist of the root of $\tree$. Recursively, for $1 \le i \le n_0$ let $P_i$ be the path from $S_{i-1}$ to $m+i$ in $\tree$, and let $P_i^*$ be $P_i$ excluding its final point $m+i$. Include the labels of the vertices of $P_i$ in $S_{i-1}$ and call the resulting set $S_i$. Then let $\sequence{v}=\sequence{v}(\tree)$ be the concatenation of $P_1^*,\ldots,P_{n_0}^*$.

{\subsection{Critical trees are not too short}
It is straightforward from the above bijection that the height of the tree is bounded from below by the length of the longest distance between two repeats in the sequence. To prove Theorem \ref{thm:height_critical0}, we will use Lemma~\ref{lem:nomasshighdegrees} to show that if we let $\mathrm{D}$ be the degree sequence of a critical Bienaym\'e tree, then the length of the longest distance between two repeats in a uniform sample from $\cS_\mathrm{D}$ is $\omega(\ln n )$ in probability. }
\begin{proof}[Proof of Theorem~\ref{thm:height_critical0}]
By Remark \ref{rk:degrees}, we can generate $\rtree_n$ in the following way. Let $D_1,\dots,D_n$ be i.i.d.\ samples from $\mu$ conditioned to sum to $n-1$. Then, let $\rseq_n=(\rseq_n(1),\ldots,\rseq_n(n-1))$ be a uniformly random sample from the set of sequences that satisfy that, for each $i\in [n]$, $i$ occurs exactly $D_i$ times. Let $\rtree_n=b(\rseq_n)$ be the tree encoded by applying the Foata--Fuchs bijection to $\rseq_N$, so that, by Proposition~\ref{prop:planetrees}, $\H(\rtree_n)$ is distributed as the height of a $\mu$-Bienaym\'e tree with size $n$. 

Fix $M>0$ large and $\varepsilon>0$ small. We are done if we show that for $n$ large enough, with probability at least $1-\varepsilon$, there exists $k \in [n-1]$ such that none of $\rseq_n(k),\rseq_n(k+1),\dots, \rseq_n(k+\lceil M\ln n \rceil )$ is a repeat. Indeed, for such $k$,  the vertices $\rseq_n(k),\rseq_n(k+1),\dots, \rseq_n(k+\lceil M\ln n \rceil )$ form a path with $\lceil M\ln n \rceil$ edges away from the root, so the height of the tree is at least $M\ln n $. 

Fix $\beta>0$ small. By Lemma~\ref{lem:nomasshighdegrees} there exists $K=K(\beta)>0$ such that for $n$ large enough,
\[\P\left(\sum_{i=1}^n D_i\I{D_i> K}\leq \beta n\right)>1-\varepsilon/2.\]

We condition on the event that $D_1,\dots,D_n$  satisfy this property. 
We call $\ell\in [n-1]$ \emph{good} if it satisfies the following two properties:
\begin{enumerate}
\item It corresponds to a vertex with degree at most $K$, or in other words, $D_{\rseq_n(\ell)}\leq K$, and 
\item It is not a repeat, or in other words, $\rseq_n(\ell)\not \in \{\rseq_n(1),\dots, \rseq_n(\ell-1)\}$.
\end{enumerate}

It is sufficient to show that with conditional probability at least $1-\varepsilon/2$ there exists a consecutive sequence of  good indices of length at least $M\ln n$. We will now show that for $\beta$ sufficiently small and $\delta=\delta(\beta)>0$ sufficiently small, it is in fact likely that there exists such a sequence amongst the first $\delta n$ elements of $\rseq_n$. 

{From now on we work conditionally given  $D_1,\dots, D_n$ and assume that $\sum_{i=1}^n D_i\I{D_i> K}\leq \beta n$.} Observe that {the law of} $(\rseq_n(1),\dots, \rseq_n(n-1))$ is {the law of a} sampling without replacement from the size-$(n-1)$ multiset in which each integer $i\in [n]$ occurs exactly $D_i$ times. Note that, by our assumption, at most $\beta n$ elements in this multiset correspond to a vertex of degree larger than $K$. Now, for a given $0<\delta<1$, for any $1\leq \ell \leq \delta n$, conditional on $\rseq_n(1),\dots, \rseq_n(\ell-1)$,  the probability that $\rseq_n(\ell)$ has degree larger than $K$ is bounded from above by  $\beta/(1-\delta)$. Furthermore, the multiset contains at most $\delta n K$ repeats of vertices in $\rseq_n(1),\dots, \rseq_n(\ell-1)$ with degree at most $K$, so the probability that $\rseq_n(\ell)$ is a repeat of a vertex with degree at most $K$ is bounded from above by $\delta K/(1-\delta)$. This implies that for all $1 \le \ell \le \delta n$, conditional on $\rseq_n(1),\dots, \rseq_n(\ell-1)$, the probability that $\ell$ is good is bounded from below by 
\[1-\frac{\beta}{1-\delta} -\frac{\delta K}{1-\delta} =:1-p ,\]
and note that we can make $p$ arbitrarily small by first choosing $\beta$ small enough (which determines the value of $K=K(\beta)$), and then choosing  $\delta$ small enough.

Therefore, still under the assumption that $\sum_{i=1}^n D_i\I{D_i> K}\leq \beta n$, it is possible to couple the random variables ${(\rseq_n(1),\dots, \rseq_n(n-1))}$ with a sequence of independent tosses of a biased coin that comes up tail with probability $p$ and head with probability $1-p$, so that for $\ell\leq \delta n$, {if the $\ell$th coin flip gives a head then $\ell$ is good.} But in this sequence, the longest consecutive sequence of heads amongst the first $\delta n$ coin flips divided by $\ln n$ converges to $-1/\ln(1-p)$ almost surely, which we can make arbitrarily big by choosing $p$ small, and in particular, choosing $p$ small enough yields that with probability at least $1-\varepsilon/2$, the longest sequence of good indices is at least $M\ln n$ for $n$ large enough, which implies the statement.
\end{proof}

\section{\bf Some critical trees are quite short and quite fat}\label{sec:quite_short_quite_fat}
In this section, we prove Theorem~\ref{thm:critical_short_trees0}, which implies that the bounds in Theorem \ref{thm:height_critical0} are optimal.

\subsection{A high degree vertex makes a tree short and fat}We will first show that the following proposition implies Theorem~\ref{thm:critical_short_trees0}, so that we only need to study the largest degree in $\rtree_n$. 

For a tree $\tree$, recall that $\Delta(\tree)$ denotes the maximal offspring of a vertex in $\tree$.

\begin{prop}\label{prop:largedegree}
Let $f(n)\to \infty$. Then, there exists a critical offspring distribution $\mu$ so that 
\[
\limsup_{n\to\infty}\P\left(\Delta(\rtree_n(\mu))>\frac{n}{f(n)}\right)=1, 
\] 
where the limit runs over all $n$ for which $\rtree_n$ is well-defined. 
\end{prop}

We will shortly prove Proposition~\ref{prop:largedegree} but we will first show how to use it to prove Theorem~\ref{thm:critical_short_trees0}. We use a stochastic domination result from \cite{AD22} for the heights of random trees under a partial order on degree sequences. This allows us to switch from a realization of the degree sequence of $\rtree_n$, say $\mathrm{D}$, to a degree sequence that is easier to work with, say $\mathrm{D}'$, that satisfies that the height of a uniform element of $\cT_{ \mathrm{D}}$ is stochastically dominated by the height of a uniform element of $\cT_{ \mathrm{D}'}$. 

For a finite set $\cA$, let $A\in_u\cA$ denote that $A$ is a uniform sample from $\cA$. 

 We say two degree sequences $\dseq=(d_1,\dots,d_n)$ and $\dseq'=(d'_1,\dots,d'_n)$ are {\em equivalent} if there exists a permutation $\eta:[n]\to [n]$ such that $(d_1,\dots,d_n)=(d'_{\eta(1)},\dots,d'_{\eta(n)})$. If this holds then ``relabelling vertices according to $\eta$'' induces a bijection from $\cT_{\dseq}$ to $\cT_{\dseq'}$. This bijection preserves the height of a tree, so for $\rtree_{\dseq}\in_u\cT_{\dseq}$ and $\rtree_{\dseq'}\in_u\cT_{\dseq'}$ it holds that $\height(\rtree_{\dseq})\overset{d}{=}\height(\rtree_{\dseq'})$. We define a partial order by  specifying a covering relation on equivalence classes of such degree sequences\footnote{
	For a partially ordered set $(\mathcal{P},\prec)$, $y \in\mathcal{P}$ covers $x \in \mathcal{P}$ if $x \prec y$ and for all $z \in \mathcal{P}$, if $x \preceq z \preceq y$ then $z \in \{x,y\}$. 
}. 

Let $\prec$ be the partial order on degree sequences of length $n$ defined by the following covering relation on equivalence classes:  for $\cD$ and $\cD'$ equivalence classes of degree sequences of length $n$ , we say $\cD$ covers $\cD'$ if there exist $\dseq=(d_1,\dots,d_n)\in \cD$ and $\sequence{d'}=(d'_1,\dots,d'_n)\in \cD'$ such that 
	\begin{enumerate}
		\item $d_1\geq d_2$;
		\item $d'_1=d_1+1$;
		\item $d'_2=d_2-1$; and
		\item $d'_k=d_k$ for $k=3,\dots,n$. 
	\end{enumerate}

In words, to obtain $\sequence{d'}$ from $\dseq$ in the definition of $\prec$, for some $a\leq b$, we replace one vertex with $a$ children and one vertex with $b$ children by a vertex with $a-1$ children and one with $b+1$ children; informally, the degrees in $\sequence{d'} $ are more skewed than the degrees in $\dseq$. 
We then have the following theorem, which states that more skewed degree sequences yield shorter trees.

\begin{thm}[Theorem 9 in \cite{AD22}]\label{thm:stochastic_order}
	Let $\dseq$ and $\dseq'$ degree sequences of length $n$ and let $\rtree_{\dseq} \in_u \cT_{\dseq}$ and $\rtree_{\dseq'}\in_u \cT_{\dseq'}$. Then, 
	\[
	 \dseq \prec \sequence{d'} \implies \height(\rtree_{\dseq})\preceq_{\mathrm{st}} \height(\rtree_{\sequence{d'}}).
	\]
\end{thm}

\begin{proof}[Proof of Theorem~\ref{thm:critical_short_trees0} {using Proposition \ref{prop:largedegree}}]
Suppose, without loss of generality, that $f(n)=o(n)$. 

Fix $n$ and fix $d^n_1,\dots,d^n_n$ in the support of $\mu$ with $\sum_{i=1}^nd^n_i=n-1$ so that $\dseq_n=(d^n_{1},\dots,d^n_{n})$ is the degree sequence of a tree. Then, for $\rtree_{\dseq_n}$ a uniform tree with degree sequence $\dseq_n$, we see that on the event that $\rtree_n$ has degree sequence $\dseq_n$, it holds that $\height(\rtree_n)\eqdist \height(\rtree_{\dseq_n})$. 

Set $\Delta_n=\max\{d^n_1,\dots,d^n_n\}$ and let $\dseq'_n=(d'^n_1,\dots,d'^n_n)$ be the degree sequence with $d'^n_1=\Delta$, $d'^n_2=\dots,d'^n_{n-\Delta_n}=1$ and $d'^n_{n-\Delta_n+1}=\dots=d'^n_n=0$. Then, we claim that $\dseq_n\preceq \dseq'_n$ in the partial order considered in Theorem \ref{thm:stochastic_order}.  To see this, suppose without loss of generality that $d^n_1=\Delta_n$. Then, one can get from $\dseq_n$ to $\dseq'_n$ (up to relabelling the vertices) by iteratively choosing a vertex with label different from $1$ with degree at least $2$, reducing its degree by $1$ and increasing the degree of a vertex with degree $0$ by $1$. At each of these steps, we move to a degree sequence that is higher in the partial ordering and the procedure terminates when all vertices but vertex $1$ have degree $1$ or $0$.  Therefore, $\dseq_n\preceq \dseq'_n$ and $\height(\rtree_{\dseq_n})\preceq_{\mathrm{st}} \height(\rtree_{\dseq'_n})$ by Theorem \ref{thm:stochastic_order}.

Any tree with degree sequence $\dseq'_n$ consists of $\Delta_n+1$ paths leaving vertex $1$, one from the root to vertex $1$ and $\Delta_n$ from vertex $1$ to the leaves. The height of the tree is the length of the path to the root plus the length of the longest path to a leaf. Since $\rtree_{\sequence{d'}_n}$ is a uniformly random tree with degree sequence $ \dseq'_n$, the lengths of these paths are distributed as a uniformly random composition of $n$ into $\Delta_n+1$ positive parts. By a stars-and-bars argument we see that the probability that the first part in such a composition has length exceeding $2n\ln(n)/\Delta_n$ is bounded from above by \[(1-\tfrac{\Delta_n}{n})^{2n\ln(n)/\Delta_n}\leq n^{-2}\] so the expected number of parts that have length greater than than $2n\ln(n)/\Delta_n$ goes to $0$ and therefore the longest part is $O_p(n\ln(n)/\Delta_n)$. It follows that if there exists a subsequence $(n_k)_{k\geq 1}$ such that {$\Delta_{n_{k}}=\omega_p(n_k/f(n_k))$}, 
then {$\height(\rtree_{\dseq'_{n_k}})=o_p(f(n_k)\ln(n_k))$}. In this case, by the stochastic domination of $\height(\rtree_{\dseq_n})$ by $\height(\rtree_{\dseq'_n})$ for each $n$, we also get that {$\height(\rtree_{\dseq_{n_k}})=o_p(f(n_k)\ln(n_k))$}.
Finally, by Proposition \ref{prop:largedegree} applied to $\sqrt{f(n)}$ there exists an offspring distribution $\mu$ and a subsequence $(n_k)_{k\geq 1}$ so that 
{$\Delta_{n_{k}}=\Omega_p(n_k/\sqrt{f(n_k)})$ and therefore $\Delta_{n_{k}}=\omega_p(n_k/f(n_k))$}, which establishes the upper bound on the height. 

Finally, since $\mathsf{Width}(\rtree_{\dseq'_{n_k}}) \ge \Delta_{n_k}$, this also establishes the lower bound on the width, and completes the proof.
\end{proof}

\subsection{{Some critical trees have a high degree vertex}}

We will now prove Proposition \ref{prop:largedegree}. We first need some additional notation and two lemmas. For an offspring distribution $\mu$ and a set $A\subset \Z_+$ with $\mu(A)>0$, we let the offspring distribution \emph{$\mu$ restricted to $A$}, denoted by $\mu^A$,  be the distribution defined by \[ \mu^A _{k}=\begin{cases}\mu_k/\mu(A)&\text{if }k\in A \\ 0&
\text{otherwise.}\end{cases}\]
 (Equivalently, for $Y$ with law $\mu$, we let $\mu^A$ be the law of $Y$ conditional on the event that $Y\in A$.) 
 
\begin{lem}\label{lem:large_degree_is_large}
Let $\mu$ be an offspring distribution and let $\Delta_n=\Delta(\rtree_n)$ be the maximal degree in a $\mu$-Bienaym\'e tree with size $n$. For $N\geq 0$, let $X^{<N}_1,X^{<N}_2,\dots,$ be i.i.d.\ samples  with law $\P(X^{<N}=i)=\mu^{\{0,1,\dots,N\}}_{i+1}$ for $i \in \{-1,\dots,N-1\}$ (i.e.\ they are distributed as $X$ conditional on the event that $X<N$). Then, for integers $n>N+1$, it holds that 
\[\P(\Delta_n\leq  N) \leq \frac{ \P(X^{<N}_1+\dots+X^{<N}_n=-1)}{n \P(X_1=N)\P( X^{<N}_2+\dots+X^{<N}_n=-N-1)}.\]
\end{lem}
\begin{proof}
First, observe that for $X_1,\dots,X_n$ i.i.d.\ samples distributed as $X$ (i.e. $\P(X_1=i)=\mu_{i+1}$), the degree sequence of a $\mu$-Bienaym\'e tree has the same law as $(X_1+1,\dots,X_n+1)$ conditional on the event that $X_1+\dots+X_n=-1$.   Then since, on the event $X_1+\dots+X_n=-1$, it holds that  $X_1,\dots, X_n$ are exchangeable, we get for any $N>0$ that  
\begin{align*} 
\P(\Delta_n> N)&\geq n \P(X_1=N, X_2,\dots, X_n<N \mid X_1+\dots+X_n=-1)\\
&= \frac{n \P(X_1=N)\P(X_2,\dots,X_n<N)\P( X^{<N}_2+\dots+X^{<N}_n=-N-1)}{\P(X_1+\dots+X_n=-1)}
\end{align*}
and 
\begin{align*} \P(\Delta_n\leq N)&= \P(X_1,\dots,X_n< N \mid X_1+\dots+X_n=-1)\\
&= \frac{\P(X_1,\dots,X_n<N)\P(X^{<N}_1+\dots+X^{<N}_n=-1)}{\P(X_1+\dots+X_n=-1)},
\end{align*}
so that 
\begin{align*}
\P(\Delta_n\leq N)&\leq \frac{\P(\Delta_n\leq  N)}{\P(\Delta_n> N) } \\
&\leq \frac{ \P(X^{<N}_1+\dots+X^{<N}_n=-1)}{n \P(X_1=N)\P( X^{<N}_2+\dots+X^{<N}_n=-N-1)}.
\end{align*}
as claimed. 
\end{proof}
Finally, for the proof of Proposition \ref{prop:largedegree}, we proceed as follows. Fix $f(n)\to \infty$.  We will explicitly construct a critical offspring distribution $\mu$ for which
\begin{equation}
\label{eq:Delta}\limsup_{n\to\infty}\P\left(\Delta_{n}>\frac{n}{f(n)}\right)=1,
\end{equation}
where we recall the notation $\Delta_n=\Delta(\rtree_n)$.

More precisely, we will construct an increasing integer sequence $(n_i)_{i \ge 0}$ with $n_0=0$, $n_1=1$, $n_2=2$, and an offspring distribution $\mu$ with mean $1$ and support $\{n_i,i \ge 0\}$. The construction will guarantee that there is another integer sequence $(n^*_i)_{i\geq 2}$ with $n_i\geq n^*_i/f(n^*_i)$ for which it holds that
$
\P(\Delta_{n^*_i}< n_i)\to 0 
$
as $i\to \infty$. 
For our construction, we first inductively build a sub-probability measure $\nu$ with support $\{n_i,i \ge 0\}$ such that, writing $p_i=\nu_{n_i}$, then $\sum_{i \ge 0} n_ip_i=1$. We then modify $\nu$ to obtain the critical probability measure $\mu$ by placing all remaining mass on $1$.

We will use Lemma~\ref{lem:large_degree_is_large} to establish \eqref{eq:Delta} for our construction. The lemma requires us to control a ratio of hitting probabilities, and for this we require the following {result}. This {result} may seem more complicated than one might expect; the reason is that the final mass at $1$ is not determined until the very end of the construction of $\mu$,  so {we need uniform bounds since }we do not yet know the precise laws of the summands appearing in the ratio in Lemma~\ref{lem:large_degree_is_large}. 
\begin{lem}\label{lem:localLDP}
Fix $K\in \N$ and $\nu_{0},\nu_1,\dots, \nu_K\ge 0$ for which $\nu_{0},\nu_1, \max\{\nu_i:i>1\}>0$ and $\nu_{0}+\nu_1+\dots+\nu_K=1-\varepsilon<1$. Suppose $\sum_{i=1}^K i\nu_i+\varepsilon=1-\delta<1$. For $c\in [0,\varepsilon]$ let $X^{(c)}$ be the random variable with support in $\{-1,0,\dots, K-1\}$ obtained by setting 
\[
\P(X^{(c)}=k)=\begin{cases} \frac{\nu_{k+1}}{ 1-\varepsilon +c }& \text{if }k\neq 0 \\ \frac{\nu_1+c}{1-\varepsilon+c}& \textrm{ if }k=0. \end{cases}
\]
Let $X^{(c)}_1,X^{(c)}_2,\dots$ be i.i.d.\ copies of $X^{(c)}$.
Then there exists $\lambda<1$ so that for all $m$ large enough, for all $c\in [0,\varepsilon]$, for all $0\leq s <\delta$ and for all $\max\{s,\delta/2\}<  r\le \delta $ such that $(\delta-r)m$ and $(\delta-s)m$ are integers,  
 \begin{align}\label{eq:exp_tails_multinom} 
 \frac{\P( X^{(c)}_1+\dots+X^{(c)}_{m}=-(\delta-r)m )}{
 \P(X^{(c)}_1+\dots+X^{(c)}_{m-1}=-(\delta-s)m)}&\leq \lambda^{(r-s)m}.
 \end{align}
\end{lem}

The proof of Lemma~\ref{lem:localLDP} involves a uniform local large deviation principle and we postpone it to after the proof of Proposition \ref{prop:largedegree}. 

\begin{proof}[Proof of Proposition \ref{prop:largedegree}]
We fix $f(n)\to \infty$ and immediately proceed with the construction described above. We start by setting $n_{0}=0$, $n_1=1$, $n_2=2$, ${\nu}_{n_0}=:p_{0}=1/2$, ${\nu}_{n_1}=:p_1=1/8$ and ${\nu}_{n_2}=:p_2=1/8$.   Now fix $k\geq 2$ and suppose we have determined $n_0,\dots,n_k$ and $p_0,\dots, p_k$ such that
\[
\varepsilon_k\coloneqq1-\sum_{i=0}^{k}p_i >0\quad \text{ and } \quad 
\delta_k\coloneqq-\sum_{i=0}^{k}(n_i-1)p_i>0
\]
(observe that $\varepsilon_2=1/4>0$ and ${\delta_2}=3/8>0$). We will choose suitable $n_{k+1}>n_k$, $n_{k+1}^*>n_{k+1}$ (whose use will appear later), $p_{k+1} \in (0,\varepsilon_{k})$ and we will set ${\nu}_{n_{k+1}}=p_{k+1}$.  We will also set ${\nu}_\ell=0$ for any $n_k<\ell<n_{k+1}$.

Specifically, observe that we are in the conditions of application of Lemma   \ref{lem:localLDP} with $K=n_k$ and $\nu_i$ for $0 \leq i\leq K$ since $\nu_0+\dots+\nu_K=1-\varepsilon_{k}<1$. Observe that  $\sum_{i=1}^K i\nu_i+\varepsilon_{k}=1-\delta_{k}<1$. Thus there exist $\lambda_{k+1} \in (0,1)$ and $n_{k+1}>n_{k}$ such that the conclusion of Lemma \ref{lem:localLDP} holds with  $m=\lceil \tfrac{1}{\delta_{k}} n_{k+1} \rceil$ and $\delta=\delta_{k}$, and such that
 \begin{enumerate} 
 \item[\optionaldesc{$(i)$}{(1)}]  $n_{k+1}-1\geq \frac{\delta}{\varepsilon}$, 
  \item[\optionaldesc{$(ii)$}{(2)}]  $ \lambda_{k+1}^{n_{k+1}-1} <\frac{\delta}{2n_{k+1}}$, \text{ and}
 \item[\optionaldesc{$(iii)$}{(3)}]   $f(\lceil \tfrac{1}{\delta} n_{k+1}\rceil)>\frac{10}{\delta_{k}}$ (which is possible because $f(n)\to \infty$).
 \end{enumerate} 
We finally define $n_{k+1}^*=\lceil\tfrac{1}{\delta_{k}}n_{k+1}\rceil$ and $ \nu_{n_{k+1}}=p_{k+1} \coloneqq \frac{\delta_{k}}{2 (n_{k+1}-1)}$. 

Note that 
\[
\frac{n_{k+1}^*}{f(n_{k+1}^*)} \le \frac{\delta_k n_{k+1}^*}{10} \le n_{k+1},
\] 
as required. It follows that to prove the proposition, it suffices to show that 
\begin{equation}
\label{eq:delta0}\P(\Delta_{n^*_i}< n_i)\to 0
\end{equation}
as $i \to \infty$.

Observe also that $\varepsilon_{k+1}=\varepsilon-p_{k+1}>0$ because by \ref{(1)} we have  $ p_{k+1}\leq \varepsilon/2$. This also implies that
\begin{equation}
\label{eq:deltak}
\delta_{k+1}=\delta_{k}- (n_{k+1}-1) p_{k+1}=\frac{\delta_{k}}{2}>0.
\end{equation}
 
 We now define $\mu$. Observe that 
\begin{equation}
\label{eq:sum}
0 \leq \sum_{i\geq 0} p_i\le 1 \qquad \textrm{and} \qquad \sum_{i\geq 0}(n_i-1) p_i =0
\end{equation}
because $\varepsilon_{k}>0$ for every $k \geq 1$ and  since $\delta_{k} \rightarrow 0$ by \eqref{eq:deltak}.  We can therefore define the offspring distribution $\mu$  by setting $\mu_{k}=\nu_{k}$ for $k \neq 1$ and $\mu_1=p_1+1-\sum_{i\geq 0}p_i$. Then
 \[
 \sum_{k=1}^{\infty} k \mu_{k}=p_1+1-\sum_{i\geq 0}p_i+\sum_{k=2}^{\infty} k \nu_{k}=1-\sum_{i\geq 0}p_i+\sum_{i=0}^{\infty} n_{i} {p_{i}}=1
 \]
 by \eqref{eq:sum}, so that  $\mu$ is critical.
 
 Now let us check that \eqref{eq:delta0} holds.  It is enough to show that for every $k \geq 1$ we have
 \begin{equation}
 \label{eq:toshowdelta}\P(\Delta_{n_{k}^*}< n_{k} ) \leq \frac{1}{n_{k}}.
 \end{equation}
 Indeed, by \ref{(3)} we have $n_{k}^*/f(n_{k}^*) \leq n_{k}$, so if \eqref{eq:toshowdelta} holds then 
\[
\P\left(\Delta_{n_{k}^*}< \frac{n_{k}^*}{f(n_{k}^*)} \right)\leq  \P(\Delta_{n_{k}^*}<  n_{k} ) \leq \frac{1}{n_{k}}  \quad  \longrightarrow \quad 0.
\]

We will use Lemmas \ref{lem:large_degree_is_large} and \ref{lem:localLDP} to  establish \eqref{eq:toshowdelta}.  Recall that by definition of $n_{k-1}$, in the notation of Lemma \ref{lem:localLDP}, by taking  $K=n_{k-1}$ and $\nu_i$ for $0 \leq i\leq K$, \eqref{eq:exp_tails_multinom} holds for $m=n_{k}^{*}$ and $\lambda=\lambda_{k}$ and $\delta=\delta_{k-1}$ since $\nu_0+\dots+\nu_K=1-\varepsilon_{k-1}$ and $\sum_{i=1}^K i\nu_i+\varepsilon_{k-1}=1-\delta_{k-1}<1$. 
Still in the notation of the latter lemma, observe that for $X^{<n_k-1}$ the random variable with law $\P(X^{<n_k-1}=k)=\mu^{\{0,\dots, n_k-1\}}_{i+1}$ (i.e. $X^{<n_k-1}$ has the law of $X$ conditional on $X<n_k-1$), it holds that 
\[
X^{<n_{k}-1}\overset{d}{=}X^{(c)}
\]
with $c=1-\sum_{i\geq 0} p_i$. We now apply  the bound \eqref{eq:exp_tails_multinom} from Lemma~\ref{lem:localLDP} with $m=n_{k}^*$ and $r,s$ defined by
\[
(\delta_{k-1}-r)n_{k}^*=1 \qquad \textrm{and}  \qquad (\delta_{k-1}-s)n_{k}^*=n_{k}
\]
by checking that  $0\leq s <\delta_{k-1}$ and that $\max\{s,\delta_{k-1}/2\}<  r\le \delta_{k-1} $.
Indeed, the first inequalities are plainly satisfied since  $n_k\leq \delta_{k-1} n^*_k$. Also $s \leq r<\delta_{k-1}$ is trivial. Finally, the inequality $\delta_{k-1}/2 \leq r$ comes from the fact that
\[
\left(\delta_{k-1}- \frac{\delta_{k-1}}{2} \right)n_{k}^*=  \frac{\delta_{k-1}}{2} n_{k}^*>1
\]
since $\delta_{k-1} n^*_k\ge n_k>2$.
 
Thus
 \[\frac{ \P(X^{<n_k-1}_1+\dots+X^{<n_k-1}_n=-1)}{\P( X^{<n_k-1}_2+\dots+X^{<n_k-1}_n=-n_k)}\leq \lambda^{n_{k}-1}.\]
 It follows from  Lemma \ref{lem:large_degree_is_large} that 
 \[\P(\Delta_{n_{k}^*}< n_{k}) {\leq \frac{\lambda^{n_{k}-1} }{{n_{k}^*} p_{k}} \leq \frac{ \lambda^{n_{k}-1} }{n_{k} p_{k}} \leq} \frac{ \lambda^{n_{k}-1} }{\delta/2}\leq \frac{1}{n_{k}},\]
 where we have used the definition of $p_{k}$ and where the {last} inequality follows from  \ref{(2)}. This entails \eqref{eq:toshowdelta} and completes the proof. 
\end{proof}

\subsection{Uniform local estimates}
We finish with the proof of Lemma \ref{lem:localLDP},  for which we need the following uniform local central limit theorem. 
\begin{lem}\label{lem:local}
Fix constants $\alpha, \beta >0$ and $0<\gamma<1/2$. Then, there exist constants $C=C(\alpha, \beta, \gamma)$ and $N=N(\alpha, \beta, \gamma)$ so that for all random variables $X$ on  $\Z$ that satisfy that 
\begin{enumerate}
\item $\P(X=-1)>\gamma$ and $\P(X=0)>\gamma$;
\item $\E{X^2}<\alpha$ and 
\item $\E{|X|^3}<\beta $,
\end{enumerate}  
for all $n>N$, for $S_n=\sum_{i=1}^n X_i$, $\mu=\E{X}$ and $\sigma^2=\E{(X-\mu)^2}$, for $\phi(t)=e^{-t^2/2}$ the  standard normal density, 
\[\sup_{k\in \Z} \left|\sqrt{n}\P\left( S_n=k\right)-\phi\left( \frac{k-n\mu}{\sigma\sqrt{n}}\right) \right|\leq \frac{C}{\sqrt{n}} .\]
\end{lem}
\begin{proof}
We will follow the proof of the local limit theorem for random variables with a finite third moment found as Theorem VII.2 in Petrov's book \cite{Petrov}, making the bounds explicit to show uniformity across the family of random variables that we consider. Fix $\alpha, \beta>0$ and $0<\gamma<1$ and let $X$ be a random variable satisfying the conditions above.
First, for $g(t)=\E{e^{itX}}$ the characteristic function of $X$, Fourier inversion yields that, for any $k$ and $x$
\[
\P\left( S_n=k\right)=\frac{1}{2\pi}\int_{-\pi}^\pi \mathrm{d}t e^{-itk}g(t)^n \qquad \text{and} \qquad e^{-x^2/2}=\frac{1}{2\pi} \int_{-\infty}^\infty \mathrm{d}t e^{-itx-t^2/2},
\]
so that, if we set $x=x_{k,n}= \frac{k-n\mu}{\sigma\sqrt{n}}$ and \[h_n(t)=\E{\exp\left(\frac{it(S_n-\mu n)}{\sigma \sqrt{n}}\right)}=\exp(-it\mu n/(\sigma \sqrt{n}))g(t/(\sigma \sqrt{n}))^n,\]
we see that 
\begin{align}
	&\left|\sqrt{n}\P\left( S_n=k\right)-\phi\left( \frac{k-n\mu}{\sigma\sqrt{n}}\right) \right|\nonumber\\
	&=\frac{1}{2\pi}\left| \int_{-\pi\sigma \sqrt{n }}^{\pi\sigma \sqrt{n }} \mathrm{d}t e^{-itx }h_{n}(t)  -\int_{-\infty}^\infty \mathrm{d}t e^{-itx-t^2/2}\right| \nonumber\\
	&\leq \int_{|t|<\sqrt{n}/4} \mathrm{d}t | h_{n}(t)-e^{-t^2/2}| +  \int_{|t|>\sqrt{n}/4} \mathrm{d}t e^{-t^2} + \int_{\sqrt{n}/4<|t|<\pi\sigma \sqrt{n }} \mathrm{d}t |h_{n}(t)| .\label{eq:OoO}
\end{align}
By \cite[Lemma V.1]{Petrov}, for $t\leq \frac{\sqrt{n}}{4} \frac{\sigma^{3}}{\E{|X-\mu|^3}}$ it holds that 
 \[| h_{n}(t)-e^{-t^2/2}| < n^{-1/2} \frac{16\E{|X-\mu|^3} }{\sigma^3} |t|^3 e^{-t^2/3}< n^{-1/2} c |t|^3 e^{-t^2/3}   \] 
 for some $c=c(\alpha, \beta, \gamma) $. The above inequality in particular holds for $t\leq \frac{\sqrt{n}}{4}$ because on probability spaces the $L_p$ norm is increasing in $p$ so ${\sigma^{3}}\le \E{|X-\mu|^3}$.

Since $|t|^3 e^{-t^2/3}$ is integrable, this implies that the first integral on the right-hand side of \eqref{eq:OoO} is $O(n^{-1/2})$ uniformly in all $X$ that satisfy our conditions. Also the second integral is $O(n^{-1/2})$ and has no dependence on the law of $X$. We now focus our attention on the final integral. 
 We see that  
 \[\int_{\sqrt{n}/4<|t|<\pi\sigma \sqrt{n }} \mathrm{d}t |h_{n}(t)|= \sqrt{n}\int_{1/(4\sigma)<|t|<\pi} \mathrm{d}t |g(t)|^n \]
 and we claim that $|g(t)|$ is bounded away from $1$ uniformly in $X$ and $1/(4\sigma)<|t|<\pi$. First observe that under our conditions, there is an $\lambda=\lambda(\alpha,\beta,\gamma)$ with $0<\lambda<\pi/2$ so that $1/(4\sigma)>\lambda$. 
 Then, we see that by $\P(X=0)>\gamma$ and  $\P(X=-1)>\gamma$, $g(t)$ can be written in the form  $g(t)=\gamma+\gamma e^{-it}+ \sum_{k \ge -1} a_{k }e^{itk}$ with $(a_{k})_{k\ge -1}$ a sequence of nonnegative real numbers summing up to $1-2\gamma$. By the triangular inequality, $|g(t)| \leq  \gamma|1+e^{- it}|+1-2\gamma= 2\gamma\cos(t/2)+1-2\gamma$. Thus  for $\lambda<|t|<\pi$ we have
 \begin{align*} |g(t)| \le 1-2\gamma(1-\cos(\lambda/2))=:1-\varepsilon.\end{align*}
We conclude that, 
 \[\sqrt{n}\int_{1/(4\sigma)<|t|<\pi} \mathrm{d}t |g(t)|^n\leq 2\pi \sqrt{n}  (1-\varepsilon)^n\leq n^{-1/2}\]
 for all $n$ large enough, uniformly in $X$. This proves the statement.
\end{proof}

\begin{proof}[Proof of Lemma \ref{lem:localLDP}]
We will prove a local large deviations result that works uniformly across the distributions defined by $c\in [0,\eps]$. For $s\in [0,\delta]$ we will use the Cramèr transform to apply a change of measure to $X^{(c)}$ so that the resulting random variable has mean $-\delta+s$ and then we apply the uniform local limit theorem we proved as Lemma \ref{lem:local} to this skewed random variable. 

Notice that for $c\in [0,\varepsilon]$, the mean of $X^{(c)}$ is  $- \frac{\delta}{1-\varepsilon+c}$, which is increasing in $c$. In particular, for any $c\in [0,\varepsilon]$, the mean of $X^{(c)}$ is bounded from above by the mean of $X^{(\varepsilon)}$, which equals $-\delta$ by our assumptions. 

Let $G(t)=G_0(t)=\frac{1}{1-\varepsilon}\sum_{i=0}^K t^{k-1} \nu_k$ be the probability generating function of $X^{(0)}$, and note that $G$ has an infinite radius of convergence since $X^{(0)}$ is bounded. Let
\[
G_c(t)=\frac{(1-\varepsilon)G(t)+{c}}{1-\varepsilon+c}
\]
be the probability generating function of $X^{(c)}$.

Observe that for every $s\in[0,\delta]$ and $c\in [0,\varepsilon]$, {the function ${bG'_c(b)}/{G_c(b)}$ is increasing in $b\in[1,\infty)$.  In addition, 
${G'_c(1)}/{G_c(1)}=\E{X^{(c)}}$ and $\lim_{b\to\infty}{bG'_c(b)}/{G_c(b)}=\max\{i:\nu_i>0\}-1>-\delta+s$.  This implies that} the equation ${bG'_c(b)}/{G_c(b)}=-\delta+s$ has a unique solution on $[1,\infty)$ which we denote by $b_{c,s}$. 
{Then, since ${bG'_c(b)}/{G_c(b)}$ is increasing in $b$, $b_{c,s}$ is increasing in $s$ and since ${bG'_c(b)}/{G_c(b)}=\frac{(1-\varepsilon)G'(b)}{(1-\varepsilon)G(b)+c}$ we see that it is decreasing in $c$ so that  $b_{c,s}$ is also increasing in $c$.} {Finally, since $1\leq b_{c,s}\leq b_{\varepsilon,\delta}<\infty$, we see that there exists a constant $C<\infty$ so that  $G_c(b_{c,s})<C$ for all $c\in [0,\varepsilon], s\in [0,\delta].$}

Now we let $\hat{X}^{(c,s)}$ be the random variable with probability generating function ${G_c(t b_{c,s})}/{G_c(b_{c,s})}$ so that by definition of $b_{c,s}$, the random variable $\hat{X}^{(c,s)}$ has mean $-\delta+s$. 
Set $S^{(c)}_n=\sum_{i=1}^n X^{(c)}_i$. Also let $\hat{X}^{(c,s)}_1, \hat{X}^{(c,s)}_2,\dots$ be i.i.d.\ copies of $\hat{X}^{(c,s)}$ and set $S^{(c,s)}_n:=\sum_{i=1}^n \hat{X}^{(c,s)}_i$. Then, we get for any $k\in \Z$ that 
\[ \P(S^{(c)}_n=k)= \frac{G_c(b_{c,s})^n }{ (b_{c,s})^k}\P(S^{(c,s)}_n=k).\]
Therefore, 
\begin{equation}\label{eq:ratio_local_LDP} \frac{\P( X^{(c)}_1+\dots+X^{(c)}_{n}= -(\delta-r)n )}{
 \P(X^{(c)}_1+\dots+X^{(c)}_{n-1}=-(\delta-s)n)}=\frac{G_c( b_{c,r} )^n }{ ( b_{c,r} )^{ -(\delta-r)n }} \frac{ (b_{c,s})^{ -(\delta-s)n}}{ G_c(b_{c,s})^{n-1}} \frac{\P(S^{(c,r)}_n=-(\delta-r)n) }{ \P(S^{(c,s)}_{n-1}=-(\delta-s)n)}.\end{equation}
 
We claim that there exists a $\lambda<1$ so that 
\begin{equation}
\label{eq:claim}
\frac{G_c( b_{c,r} ) }{ ( b_{c,r} )^{ -\delta+r }} <\lambda^{r-s} \frac{ G_c(b_{c,s})}{ (b_{c,s})^{ -\delta+s}}
\end{equation}
for all $s\geq 0$ and for all $r$ with $\max\{s,\delta/2\}<  r\le \delta $ and for all $c\in [0,\varepsilon]$. 
To see this, note that 
\[ \frac{d}{dr} \ln\left(\frac{G_c( b_{c,r} ) }{ ( b_{c,r} )^{-\delta+r } }\right)=-\ln( b_{c,r} ),\]
so since $ b_{c,r} \ge1$ for all $c$ and $r$ that we consider, 
\begin{align*}
\ln\left(\frac{G_c(b_{c,s}) }{ (b_{c,s})^{-\delta+s }}\right)-\ln\left(\frac{G_c( b_{c,r} ) }{ ( b_{c,r} )^{-\delta+r }}\right)&=\int_s^r \mathrm{d}t \ln( b_{c,t} )\\
&>\int_{(r+s)/2}^r \mathrm{d}t \ln( b_{c,t} )\\
&>\frac{r-s}{2}\ln(b_{c,(r+s)/2})\\
&>\frac{r-s}{2}\ln(b_{\varepsilon,\delta/4})
\end{align*}
since $(r+s)/2>\delta/4$  by our assumptions and since $b_{c,s}$ is increasing in $c$ and $s$. This implies the claim \eqref{eq:claim} with $\lambda=(b_{\varepsilon,\delta/4})^{-1/2}<1$.

Plugging this into \eqref{eq:ratio_local_LDP} yields that
\begin{align*}\frac{\P( X^{(c)}_1+\dots+X^{(c)}_{n}= -(\delta-r)n )}{
 \P(X^{(c)}_1+\dots+X^{(c)}_{n-1}=-(\delta-s)n)}\leq &G_c(b_{c,s}) \lambda^{(r-s)n} \frac{\P(S^{(c,r)}_n=-(\delta-r)n) }{ \P(S^{(c,s)}_{n-1}=-(\delta-s)n)}\\
 \leq&C \lambda^{(r-s)n} \frac{\P(S^{(c,r)}_n=-(\delta-r)n) }{ \P(S^{(c,s)}_{n-1}=-(\delta-s)n)}
 \end{align*}
We will control the ratio on the right-hand side using Lemma \ref{lem:local}. 

We claim that we can apply Lemma \ref{lem:local} to the family of random variables $\{\hat{X}^{(c,s)}:c\in [0,\varepsilon], s\in [0,\delta]$. Indeed, since $\hat{X}^{(c,s)}$ has support in $\{-1,0,\dots, K-1\}$, we have $\E{(\hat{X}^{(c,s)})^2}<K^2$ and $\E{|\hat{X}^{(c,s)}|^3}<K^3$, so we can set $\alpha=K^2$ and $\beta=K^3$ in the conditions of Lemma \ref{lem:local}. Then, note that
\[\P(\hat{X}^{(c,s)}=-1)=\frac{\P(X^{(c)}=-1)}{b_{c,s}G_c(b_{c,s})}=\frac{\nu_0}{ (1-\varepsilon+c)b_{c,s}G_c(b_{c,s})}\geq  \frac{\nu_0}{ b_{\varepsilon,\delta} C}=:\gamma_0>0.\]
 Also,
\[\P(\hat{X}^{(c,s)}=0)=\frac{\P(X^{(c)}=0)}{G_c(b_{c,s})}=\frac{\nu_1+c}{ (1-\varepsilon+c)G_c( b_{c,s} )  }\geq  \frac{\nu_1}{C}=:\gamma_1>0.\]
Therefore we may set $\gamma=\min\{\gamma_0,\gamma_1\}$ in the conditions of Lemma \ref{lem:local}. Let $(\hat{X}^{(c,s)}_i,i \ge 1)$ be i.i.d.\ copies of $\hat{X}^{(c,s)}$ and set $S^{(c,s)}_n:=\sum_{i=1}^n \hat{X}^{(c,s)}_i$. Then the lemma yields constants $A>0$ and and $N_0>0$ so that for all $n\geq N_0$, for $\sigma^{(c,s)}$ the standard deviation of $\hat{X}^{(c,s)}$,
\[ \sup_{c\in [0,\varepsilon], s\in[0,\delta], k\in \Z} \left|\sqrt{n}\P\left( S^{(c,s)}_n=k\right)-\phi\left( \frac{k+n(\delta-s)}{\sigma^{(c,s)}\sqrt{n}}\right) \right|\leq \frac{A}{\sqrt{n}}.\]
Now observe that there exists $a>0$ so that $\sigma^{(c,s)}\in [1/a,a]$ for all $c\in [0,\varepsilon], s\in[0,\delta]$. Therefore, we see that there exist $b>0$ and $N_1>0$ so that for all $n\ge N_1$, $c\in [0,\varepsilon]$, and $s\in[0,\delta]$, for all $k\in \Z$ for which $|k+n(\delta-s)|\le 1 $, 
\[ \frac{1}{b\sqrt n} \leq \P\left( S^{(c,s)}_n=k\right)\leq \frac{b}{\sqrt{n}}.\]
Therefore, 
\begin{align*}
 \frac{\P(S^{(c,r)}_n=-(\delta-r)n) }{ \P(S^{(c,s)}_{n-1}=-(\delta-s)n)}\leq b^2\sqrt{n/(n-1)}\leq 2b^2
\end{align*}
for $n\geq 2$. This is also a constant, so, by possibly making $\lambda$ larger (but still less than $1$, and independently of $c$,$r$ and $s$), we see that there is $N>0$ such that for all $n \ge N$, for all $c\in [0,\varepsilon]$, for all $0\leq s<\delta$, for all $\max\{s,\delta/2\}<r\leq \delta$, 
\[
\frac{\P( X^{(c)}_1+\dots+X^{(c)}_{n}= -(\delta-r)n )}{
 \P(X^{(c)}_1+\dots+X^{(c)}_{n-1}=-(\delta-s)n)}\leq   \lambda^{(r-s)n} 
\]
as claimed.
\end{proof}

\section{\bf Future work}\label{sec:future}
We conclude by mentioning several natural directions for future research.
\begin{enumerate}
\item We are curious to know whether, in Theorem~\ref{thm:critical_short_trees0}, the $\limsup$s can be replaced by limits. Equivalently, we would like to know whether for any $f(n) \to \infty$, there exists a critical offspring distribution such that $\H(\rtree_n)=o_{\p}(f(n)\ln n)$ and 
$\mathsf{Width}(\rtree_n)=\omega_{\P}(n/f(n))$. 
\item What are the possible asymptotic behaviours of $\H(\rtree_n)$? For which functions $f(n)=\omega(1)$ can we find a critical offspring distribution for which $(\H(\rtree_n)/(f(n)\ln n),n \ge 1)$ is a tight sequence, or converges in probability to a constant?
\item The paper \cite{Add19} asked whether it is always the case that
$\H(\rtree_n)\cdot \mathsf{Width}(\rtree_n)=O_{\p}(n\ln n)$, and constructed examples where $\H(\rtree_n)\cdot \mathsf{Width}(\rtree_n)=\Theta_{\p}(n\ln n)$. Section~\ref{ssec:examples} provides some further examples where $\H(\rtree_n)\cdot \mathsf{Width}(\rtree_n)=\Theta_{\p}(n\ln n)$,  and in Corollary \ref{cor:hw} we show that answer to the above question is ``yes'' for the class of offspring distributions satisfying  \eqref{eq:hypmu}. We expect that the same holds for all critical offspring distributions, and suspect that it holds for all offspring distributions.
\item Given a slowly varying function $\Lambda(n)=\omega(\ln n)$, is it possible to find an offspring distribution satisfying \eqref{eq:hypmu} such that $\H(\rtree_{n})/\Lambda(n) \rightarrow 1$ in probability?
\end{enumerate}

\bibliographystyle{plainnat}

\end{document}